\documentclass{article}

\pdfoutput=1

\usepackage{amsmath, amssymb, amsthm, enumerate, fullpage}
\usepackage{verbatim, enumitem, bbm, etoolbox}

\newcommand{\m}{\mathcal}
\renewcommand{\b}{\mathbb}

\apptocmd{\lim}{\limits}{}{}

\newcommand{\acl}{\textrm{acl}}

\newcommand{\tp}{\textrm{tp}}
\renewcommand{\o}{\overline}

\newcommand{\model}{\models}

\theoremstyle{plain}
\newtheorem{thm}{Theorem}
\newtheorem{theorem}[thm]{Theorem}

\newtheorem{lemma}[thm]{Lemma}
\newtheorem{cor}[thm]{Corollary}
\newtheorem{prop}[thm]{Proposition}
\newtheorem{fact}[thm]{Fact}

\numberwithin{thm}{section}

\numberwithin{subcase}{case}

\theoremstyle{definition}
\newtheorem{definition}[thm]{Definition}

\newtheorem{example}[thm]{Example}
\newtheorem{question}[thm]{Question}
\newtheorem{notation}[thm]{Notation}

\def\P{{\cal P}}

\def\Ind{\setbox0=\hbox{$x$}\kern\wd0\hbox to 0pt{\hss$\mid$\hss}
	\lower.9\ht0\hbox to 0pt{\hss$\smile$\hss}\kern\wd0}

\def\Notind{\setbox0=\hbox{$x$}\kern\wd0\hbox to 0pt{\mathchardef
		\nn=12854\hss$\nn$\kern1.4\wd0\hss}\hbox to 0pt{\hss$\mid$\hss}\lower.9\ht0
	\hbox to 0pt{\hss$\smile$\hss}\kern\wd0}

\newcommand{\Mod}{\textrm{Mod}}
\newcommand{\iso}{\cong}

\newcommand{\dom}{\textrm{dom}}

\newcommand{\hceq}{\sim_{_{\!\!\HC}}}
\newcommand{\hcleq}{\leq_{_{\!\HC}}}
\newcommand{\hclt}{<_{_{\!\HC}}}

\newcommand{\boreleq}{\sim_{_{\!\!B}}}
\newcommand{\borelleq}{\leq_{_{\!B}}}
\newcommand{\borellt}{<_{_{\!B}}}

\newcommand{\REF}{\textrm{REF}}

\newcommand{\K}{\textrm{K}}
\newcommand{\TK}{\textrm{TK}}

\newcommand{\ptl}{\textrm{ptl}}
\newcommand{\ext}{\ptl}

\newcommand{\sat}{\textrm{sat}}

\newcommand{\eni}{\textrm{eni}}
\newcommand{\ENIDOP}{\textrm{ENI-DOP}}
\newcommand{\ENINDOP}{\textrm{ENI-NDOP}}
\newcommand{\HC}{\textrm{HC}}
\newcommand{\CSS}{\textrm{CSS}}
\newcommand{\css}{\textrm{css}}

\newcommand{\forces}{\Vdash}
\newcommand{\Aut}{\textrm{Aut}}
\newcommand{\val}{\textrm{val}}

\newenvironment{claim}[1]{\smallskip\par\noindent\underline{Claim:}\space#1}{}

\newenvironment{claim1}[1]{\smallskip\par\noindent\underline{Claim 1:}\space#1}{}
\newenvironment{claim2}[1]{\smallskip\par\noindent\underline{Claim 2:}\space#1}{}
\newenvironment{claimproof}[1]{\par\noindent\underline{Proof:}\space#1}{\leavevmode\unskip\penalty9999 \hbox{}\nobreak\hfill\quad\hbox{$\square$}}

\def\F{{\cal F}}
\def\V{\b V}
\def\LL{\b L}
\def\N{{\cal W}}

\def\PP{\b P}
\def\QQ{\b Q}
\def\RR{\b R}

\def\phi{\varphi}
\def\abar{{\overline{a}}}
\def\bbar{{\overline{b}}}
\def\cbar{{\overline{c}}}

\def\alphabar{{\overline{\alpha}}}
\def\xbar{{\overline{x}}}
\def\tp{{\rm tp}}
\def\lg{{\rm lg}}
\def\<{\langle}
\def\>{\rangle}
\def\acl{\textrm{acl}}
\def\mult{\textrm{mult}}
\def\inf{\textrm{inf}}
\def\bin{\textrm{bin}}
\def\fin{\textrm{fin}}
\def\qf{\textrm{qf}}

\newcommand\myrestriction{\mathord\restriction}
\def\mr#1{\myrestriction_{#1}}

\begin{document}	

	\bibliographystyle{plain}
	
	\author{Douglas Ulrich, Richard Rast, and Michael C.\ Laskowski\thanks{All three authors were partially supported by NSF Research Grant DMS-1308546.}
	 \\ Department of Mathematics\\University of Maryland, College Park}
	\title{Borel Complexity and Potential Canonical Scott Sentences}
	\date{\today} 
	
	\maketitle

\begin{abstract}	
	We define and investigate $HC$-forcing invariant formulas of set theory, whose interpretations in the hereditarily countable sets are well behaved under forcing extensions. 
	 This leads naturally to a notion of cardinality $||\Phi||$ for sentences $\Phi$  of $L_{\omega_1,\omega}$, which counts the number of sentences of $L_{\infty,\omega}$ that, in some forcing extension, become a canonical Scott sentence of a model of $\Phi$.  
	 We show this cardinal bounds the complexity of $(\Mod(\Phi), \iso)$, the class of models of $\Phi$ with universe $\omega$, by proving that $(\Mod(\Phi),\iso)$ is not Borel reducible to $(\Mod(\Psi),\iso)$ whenever $||\Psi|| < ||\Phi||$. 
	Using these tools, 
	we analyze the complexity of the class of countable models of four complete, first-order theories $T$ for which  $(\Mod(T),\iso)$ is properly analytic,  yet admit very different behavior. We prove that both `Binary splitting, refining equivalence relations' and Koerwien's example \cite{KoerwienExample} of an eni-depth 2, $\omega$-stable theory 
have  $(\Mod(T),\iso)$ non-Borel, yet neither is  Borel complete.   We give a  slight modification of Koerwien's example that also is $\omega$-stable, eni-depth 2, but is Borel complete.
Additionally, we prove that $I_{\infty,\omega}(\Phi)<\beth_{\omega_1}$ whenever  $(\Mod(\Phi),\iso)$ is Borel.
\end{abstract}

\section{Introduction}	
In their seminal paper \cite{FriedmanStanleyBC}, Friedman and Stanley define and develop a notion of {\em Borel reducibility} among classes $C$ of structures with universe $\omega$ in a fixed, countable language $L$ that are Borel and invariant under permutations of $\omega$.  It is well known (see e.g., \cite{KechrisDST} or \cite{GaoIDST}) that such classes are of the form $\Mod(\Phi)$, the set of models of $\Phi$ whose universe is precisely $\omega$ for some sentence $\Phi\in L_{\omega_1,\omega}$.  A {\em Borel reduction} is a Borel function $f:\Mod(\Phi)\rightarrow \Mod(\Psi)$ that satisfies $M\iso N$ if and only if $f(M)\iso f(N)$.  One says that $\Mod(\Phi)$ is {\em Borel reducible} to $\Mod(\Psi)$, written $(\Mod(\Phi),\iso)\le_B (\Mod(\Psi),\iso)$ or more typically $\Phi \leq_B \Psi$,  if there is a Borel reduction $f:\Mod(\Phi)\rightarrow \Mod(\Psi)$; and the two classes are {\em Borel equivalent} if there are Borel reductions in both directions. 
As Borel reducibility is transitive, $\le_B$ induces a pre-order on $\{\Mod(\Phi):\Phi\in L_{\omega_1,\omega}\}$. In \cite{FriedmanStanleyBC}, Friedman and Stanley show that among Borel invariant classes, there is a maximal class with respect to $\le_B$.  We say $\Phi$ is {\em Borel complete} if it is in this maximal class.  Examples include the theories of graphs, linear orders, groups, and fields.

It is easily seen that for any $\Phi\in L_{\omega_1,\omega}$, the isomorphism relation on $\Mod(\Phi)$ is analytic, but in many cases it is actually Borel.
The isomorphism relation on any Borel complete class is properly analytic (i.e., not Borel) but prior to this paper there were few examples known
of classes $\Mod(\Phi)$ where the isomorphism relation is properly analytic, but not Borel complete.  Indeed, the authors are only aware of the example of abelian $p$-groups
that first appeared in \cite{FriedmanStanleyBC}.  The class of abelian $p$-groups is  expressible by a sentence of $L_{\omega_1,\omega}$, but no first-order example of this phenomenon was known.

To date, the study of classes $(\Mod(\Phi),\iso)$ is much more developed when the isomorphism relation is Borel than when it is not.
Here, however, we define the set of {\em short} sentences $\Phi$ (see Definition~\ref{shortdef}), 
that properly contain the set of sentences for which isomorphism is Borel but exclude the Borel complete sentences, and develop criteria for concluding
$\Mod(\Phi)\not\le_B \Mod(\Psi)$ among pairs $\Phi,\Psi$ of short sentences.  
Furthermore, Theorem~\ref{translate} asserts that there can never be a Borel reduction from a non-short $\Phi$ into a short $\Psi$.  
From this technology, we are able to discern more about $\Mod(\Phi)$ when $\iso$ is Borel.  For example, Corollary~\ref{Borelisshort} asserts that
\begin{quotation}
	If $\iso$ is Borel on $\Mod(\Phi)$, $I_{\infty,\omega}(\Phi)<\beth_{\omega_1}$.
\end{quotation}
That is, among models of $\Phi$ of any cardinality, there are fewer than $\beth_{\omega_1}$ pairwise
back-and-forth inequivalent models.



In this paper we work in ZFC, and so formalize all of model theory within ZFC. We do this in any reasonable way. For instance, countable languages $L$ are construed as being  elements of $\HC$ (the set of hereditarily countable sets) and the set of sentences of $L_{\omega_1, \omega}$ form a subclass  of $\HC$.  Moreover, basic operations, such as describing the set of 
subformulas of a given formula, are absolute class functions of $\HC$.  As well, if  $L \in \HC$ is a language, then for every $L$-structure $M$, the canonical Scott sentence $\css(M)$ is in particular a set; and if $M$ is countable then $\css(M) \in L_{\omega_1,\omega}$ is an element of $\HC$. Moreover the class function $\css$ is highly absolute.
 
One of our central ideas is the following. Given $\Phi \in L_{\omega_1,\omega}$, let $\CSS(\Phi) \subseteq \HC$ denote the set $\{\css(M): M\in \Mod(\Phi)\}$. Then any Borel map $f: \Mod(\Phi) \rightarrow \Mod(\Psi)$ induces an $\HC$-definable  function $f^*: \CSS(\Phi) \rightarrow \CSS(\Psi)$. This leads us to the investigation of definable subclasses of $\HC$ and definable maps between them. 
We begin by restricting our notion of classes to those definable by {\em $\HC$-forcing invariant} formulas (see Definition~\ref{HCdef}).
Three straightforward consequences of the Product Forcing Lemma 
allow us to prove that these classes are well-behaved.
It is noteworthy that we do not define $\HC$-forcing invariant formulas syntactically.  Whereas it is true that every $\Sigma_1$-formula is $\HC$-forcing invariant,
determining precisely which classes are $\HC$-forcing invariant depends on our choice of $\V$.  

The second ingredient of our development is that every set $A$ in $\V$ is {\em potentially in $\HC$}, i.e., there is a forcing extension $\V[G]$ of $\V$
(indeed a Levy collapse suffices) such that $A\in \HC^{\V[G]}$.  Given an $\HC$-forcing invariant $\phi(x)$, we define $\phi_\ext$ -- that is, the {\em potential solutions to $\phi$} --   to be those $A\in\V$ for which
$(\HC^{\V[G]},\in)\models\phi(A)$ whenever $A\in \HC^{\V[G]}$.  The definition of $\HC$-forcing invariance makes this notion well-defined.
Thus, given a sentence $\Phi$, one can define $\CSS(\Phi)_\ext$, which should be read as the class of `potential canonical Scott sentences' i.e., 
the class of all $\phi\in L_{\infty,\omega}$ such that in some forcing extension $\V[G]$, $\phi$ is the canonical Scott sentence of some countable model of $\Phi$.
We define $\Phi$ to be {\em short} if $\CSS(\Phi)_\ext$ is a {\bf set} as opposed to a proper class and define the \emph{potential cardinality} of $\Phi$, denoted $||\Phi||$, to
be the (usual) cardinality of $\CSS(\Phi)_\ext$ if $\Phi$ is short, or $\infty$ otherwise.    By tracing all of this through, with Theorem~\ref{translate}(2), we see that 
\begin{quotation}  If
$||\Psi||<||\Phi||$, then there cannot be a Borel reduction $f:\Mod(\Phi)\rightarrow \Mod(\Psi)$.  
\end{quotation}


Another issue that is raised by our investigation is a comparison of the class the potential canonical Scott sentences $\CSS(\Phi)_\ext$ with the class $\CSS(\Phi)_\sat$, consisting of
all sentences of $L_{\infty,\omega}$ that are canonical Scott sentences  of some model $M\models\Phi$  with $M\in\V$.
Clearly, the latter class is contained in the former, and we call $\Phi$ {\em grounded} (see Definition~\ref{grounded}) if equality holds.  We show that the incomplete theory $\REF$ of
refining equivalence relations is grounded.  By contrast, the theory $\TK$,  defined
in Section~\ref{OmegaStableSection}, is a  complete,
$\omega$-stable theory for which $|\CSS(\TK)_\sat|=\beth_2$, while $\CSS(\TK)_\ext$ is a proper class.  



Sections~\ref{CompactGroupSection}-\ref{OmegaStableSection} apply this technology.  Section~\ref{CompactGroupSection} discusses continuous actions by compact groups on Polish spaces.  In addition to being of interest in its own right, the results there are also used in Section~\ref{OmegaStableSection}.  Sections~\ref{REFsection} and \ref{OmegaStableSection} discuss four complete first-order theories that are not very complicated stability-theoretically, yet the isomorphism relation is
properly analytic in each case.  We summarize our findings by:

\medskip

{\bf $\REF({\bf inf})$} is the theory of `infinitely splitting, refining equivalence relations'.  Its language is $L=\{E_n:n\in\omega\}$. It asserts that each $E_n$ is an equivalence relation,  $E_0$ consists of a single class, each $E_{n+1}$ refines $E_n$, and each $E_n$-class is partitioned into infinitely many $E_{n+1}$-classes.  $\REF(\inf)$ is one of the standard examples of a stable, unsuperstable theory.  Then:

\begin{itemize}
\item  $\REF(\inf)$ is Borel complete, in fact, it is $\lambda$-Borel complete for all infinite $\lambda$  (see Definition~\ref{lambdaBC});
\item  Therefore, $\REF(\inf)$ is not short;
\item  Therefore, $\iso$ is not Borel.
\item  $\REF(\inf)$ is grounded, i.e., $\CSS(\REF(\inf))_\sat=\CSS(\REF(\inf))_\ext$.
\end{itemize}

%

\medskip

{\bf $\REF({\bf bin})$} is the theory of `binary splitting, refining equivalence relations'.  The language is also $L=\{E_n:n\in\omega\}$.  The axioms of $\REF(\bin)$ assert that each $E_n$ is an equivalence relation, $E_0$ is trivial, each $E_{n+1}$ refines $E_n$, and each $E_n$-class is partitioned into exactly two $E_{n+1}$-classes.  $\REF(\bin)$ is superstable (in fact, weakly minimal) but is not $\omega$-stable.  Then:

\begin{itemize}
\item  $T_2$ (i.e., `countable sets of reals') is Borel reducible into $\Mod(\REF(\bin))$, but

\item  $\REF(\bin)$ is short, with $||\REF(\bin)|| = \beth_2$.
\item  Therefore, $T_3$ (i.e., `countable sets of countable sets of reals') is not Borel reducible into $\Mod(\REF(\bin))$;
\item Therefore, $\REF(\bin)$ is not Borel complete.

\item  $\iso$ is not Borel.

\item  $\REF(\bin)$ is grounded.

\end{itemize}

\medskip

{\bf $\K$} is the Koerwien theory, originating in \cite{KoerwienExample} and defined in Section~\ref{OmegaStableSection}.  Koerwien proved that $\K$ is complete, $\omega$-stable,
eni-NDOP,  and of eni-depth 2.
Then:

\begin{itemize}
\item  $T_2$ (i.e., `countable sets of reals') is Borel reducible into $\Mod(\K)$, and
\item  $\K$ is short, with $||\K||=\beth_2$.
\item  Therefore, $T_3$ (i.e., `countable sets of countable sets of reals') is not Borel reducible into $\Mod(\K)$;
\item Therefore, $\Mod(\K)$ is not Borel complete;
\item  Nonetheless, $\iso$ is not Borel (this was proved by Koerwien in \cite{KoerwienExample}).

\item  Whether $\K$ is grounded or not remains open.
\end{itemize}
{\bf $\TK$} is a `tweaked version of $\K$' and is also defined in Section~\ref{OmegaStableSection}.  $\TK$ is also complete, $\omega$-stable, eni-NDOP, of $\eni$-depth 2,
and is very much like the theory $\K$, however the automorphism groups of models of $\TK$ induce a more complicated group of elementary permutations of
$\acl(\emptyset)$ than do the automorphism groups of models of $\K$.  Then:

\begin{itemize}
\item  $\TK$ is Borel complete, hence not short, hence $\cong$ is is not Borel.

\item  $\TK$ is not grounded; in fact $\CSS(\TK)_\ext$ is a proper class, while $|\CSS(\TK)_\sat|=\beth_2$.
\item Since $|\CSS(\TK)_{\sat}| = \beth_2$, $\TK$ is not $\lambda$-Borel complete for sufficiently large $\lambda$.
\end{itemize}

Many of the ideas for this arose from the first author, who read and abstracted ideas about absolutely ${\bf \Delta^1_2}$-formulas in \cite{FriedmanStanleyBC} and ideas that were discussed in
Chapter~9 of \cite{hjorthBook}.   Only recently, the authors became aware of work on `pinned equivalence relations' as surveyed in e.g., \cite{ZapletalForcingBorel}.
The whole of this paper was written independently of the development there.  
In terms of notation, $\HC$ always denotes $H(\aleph_1)$, the set of all sets whose transitive closure is countable.  For a forcing extension $\V[G]$ of $\V$,
$\HC^{\V[G]}$ denotes those sets that are hereditarily countable in $\V[G]$.  Throughout, we only consider set forcings,
so when we write `for $\V[G]$ a forcing extension of $\V\dots$' we are quantifying over all (set) forcing notions $\PP\in\V$ and all $\PP$-generic filters $G$.

Finally, we are grateful to the referee for their very careful reading and their finding simplifications to some of our arguments.

\section{A notion of cardinality for some classes of $\HC$}\label{formulaSection}

We develop a notion of cardinality on certain well-behaved subclasses  of $\HC$ in terms of the existence or non-existence of
certain injective maps.  
Behind the scenes, we rely heavily on the fact that all sets $A$ in $\b V$ are `potentially' elements of $\HC$, the set of hereditarily countable sets.
For example, if $\kappa$ is the cardinality of the transitive closure of $A$
and we take $\PP$ be the Levy collapsing poset $\textrm{Coll}(\kappa^+,\omega_1)$ that collapses $\kappa^+$ to $\omega_1$, then for any choice $G$ of a generic filter,
$A\in \HC^{\V[G]}$.

\subsection{$\HC$-forcing invariant formulas}

We begin with our principal definitions.

\begin{definition}  \label{HCdef}
Suppose $\phi(x)$ is any formula of set theory, possibly with a hidden parameter from $\HC$.  
\begin{itemize}
\item  $\phi(\HC)=\{a\in \HC:(\HC,\in)\models \phi(a)\}$.
\item  If $\V[G]$ is a forcing extension of $\b V$, then $\phi(\HC)^{\V[G]}=\{a\in \HC^{\V[G]}:\V[G]\models 
\hbox{`$a\in \phi(\HC)$'}\}$.
\item  $\phi(x)$ is
{\em $\HC$-forcing invariant}
if, for every twice-iterated forcing extension $\V[G][G']$,
\[\phi(\HC)^{\V[G][G']}\quad \cap\quad  \HC^{\V[G]}\quad = \quad \phi(\HC)^{\V[G]}\]
\end{itemize}
\end{definition}

The reader is cautioned that when computing $\phi(\HC)^{\V[G]}$, the quantifiers of $\phi$ range over $\HC^{\V[G]}$ as opposed to the whole of $\V[G]$.
Visibly, the class of $\HC$-forcing invariant formulas is closed under boolean combinations.
Note that by Shoenfield's Absoluteness Theorem, e.g., Theorem~25.20 of \cite{JechSetTheory}, any $\Sigma^1_2$ subset of $\RR$ is $\HC$-forcing invariant. There is also the closely related L\'evy Absoluteness Principle, which has various forms (e.g., Theorem~9.1 of \cite{barwise2} or Section~4 of \cite{LevyAbs}); we give a version more convenient to us.

\begin{lemma}  \label{sigma1} If $\V[G]$ is any forcing extension, and if $\phi(x)$ is a $\Sigma_1$ formula of set theory, then for every $a\in\HC$, $\HC \models \phi(a)$
 if and only if $\HC^{\V[G]} \models \phi(a)$. In particular,  $\Sigma_1$-formulas are $\HC$-forcing invariant.
 \end{lemma}
 
 \begin{proof}
This can be proved using Shoenfield's Absoluteness Theorem, as in Exercise~25.4 of \cite{JechSetTheory} (which readily relativizes to allow a parameter from $\HC$). 
More directly, let  $\phi(x)$ be $\exists y \psi(y, x)$,  where $\psi(y, x)$ is $\Delta_0$.   Let $\V[G]$ be any forcing extension of $\V$ and assume $\HC^{\V[G]}\models\phi(a)$
for some $a\in \HC^{\V}$.

In $\V[G]$, there is a countable, transitive set $M$ containing $a$ and a witness $b$ to $\phi(a)$.  Choose a bijection $j:M\rightarrow\omega$ with $j(b)=1$ and $j(a)=0$.
The image of $\in$ restricted to  $M$ is a well-founded, extensional relation $E$ such that $(\omega,E)\models\psi(1,0)$.

Let $\kappa = (\omega_1)^{\V[G]}$.  Working in $\V$, there is  an $\omega$-tree $T$  whose ill-foundedness witnesses the existence of $E$ as above. Without going into all the details, $T$ consists of certain pairs ($\sigma(x_0, \ldots, x_{n-1}), r)$, where $\sigma(x_0, \ldots, x_{n-1})$ is a formula of set theory in the listed variables, $r: n \to \kappa$
is a function, and we arrange that when $\<(\sigma_n, r_n):n\in\omega\>$  is a branch through $T$, then  $p:= \bigcup_n \sigma_n$ is a complete type in the variables 
$(x_n:n\in\omega)$, such that the induced model $(\omega, E')$ is extensional, and $(\omega, E') \models \psi(1, 0)$, and $0$ codes $a$, and $r := \bigcup_n r_n: \omega \to \kappa$ is a rank function witnessing that $(\omega, E')$ is well-founded.

As $(\omega,E)$ is well-founded in $\V[G]$,  $T$ is ill-founded in $\V[G]$.  As well-foundedness is $\Delta_1$, this implies $T$ is ill-founded in $\V$ as well.
Thus, $\psi(x,a)$ has a witness in $\HC^\V$.
\end{proof}

For more complicated formulas, whether or not
$\phi(x)$ is $\HC$-forcing invariant or not may well depend on the choice of set-theoretic universe.  For example, consider the formula $\phi(x):= (x = \emptyset) \vee (\V\neq \LL)$. Then $\phi(\HC)$ is equal to $\{\emptyset\}$ if $\HC \subseteq \mathbb{L}$, and $\phi(\HC) = \HC$ otherwise. Because `$\HC\not\subseteq \mathbb{L}$' holding in $\V$
implies that it holds in any forcing extension $\V[H]$,  it follows that the formula  $\phi(x)$ is $\HC$-forcing invariant if and only if $\HC \not \subseteq \mathbb{L}$. 


Before continuing, we state three set-theoretic lemmas that form  the lynchpin of our development. Lemma~\ref{product1} is a simple consequence of our definitions.
Lemma~\ref{product2} is  well-known.   It is mentioned  in the proof of Theorem 9.4 in \cite{hjorthBook}; a full proof is given in the more recent \cite{kaplan}. Lemma~\ref{product3} is just a rephrasing of Lemma~\ref{product2}. The key tool for all of these lemmas is the Product Forcing Lemma, see e.g., Lemma~15.9 of \cite{JechSetTheory}, 
which states that given any $\PP_1\times\PP_2$-generic filter $G$,
if $G_\ell$ is the projection of $G$ onto $\PP_\ell$, then $G=G_1\times G_2$, each $G_\ell$ is $\PP_\ell$-generic, and
$\V[G]=\V[G_1][G_2]=\V[G_2][G_1]$ (i.e., $G_\ell$ meets every dense subset of $\PP_{\ell}$ in $\V[G_{3-\ell}]$).

\begin{lemma}  \label{product1}
 Suppose $\PP_1,\PP_2\in \b V$  are notions of forcing and $\phi(x)$ is $\HC$-forcing invariant (possibly with a hidden parameter from $\HC$).
If $A\in V$ and, for $\ell=1,2$, $H_\ell$ is $\PP_\ell$-generic and $V[H_\ell]\models A\in \HC$, then
$\V[H_1]\models A\in\phi(\HC)$ if and only if $\V[H_2]\models A\in \phi(\HC)$.  (The filters $H_1$ and $H_2$ are not assumed to be mutually generic.)
\end{lemma}

\begin{proof} Assume this were not the case.  By symmetry, choose $p_1\in H_1$ such that $p_1\forces \check{A}\in\phi(\HC)$, and
choose $p_2\in H_2$ such that
$p_2\forces \check{A}\in \HC\wedge\check{A}\not\in\phi(\HC)$.  Let $G$ be a $\PP_1\times \PP_2$-generic filter with $(p_1,p_2)\in G$.  
Write $G=G_1\times G_2$, hence $\V[G]=\V[G_1][G_2]=\V[G_2][G_1]$.
As $p_1\in G_1$ and $p_2\in G_2$, we have $\V[G_1]\models  A\in\phi(\HC)$ and
$\V[G_2]\models  A \in \HC \land A\not\in\phi(\HC)$.  But applying the $\HC$-forcing invariance of $\phi$ twice, we get that $A \in \phi(\HC)^{\mathbb{V}[G_1]}$ iff $A \in \phi(\HC)^{\mathbb{V}[G_1][G_2]}$ iff $A \in \phi(\HC)^{\mathbb{V}[G_2]}$, a contradiction.
\end{proof}

\begin{lemma}  \label{product2}  Suppose $\PP_1$ and $\PP_2$ are both notions of forcing in $\b V$.
If $G$ is a $\PP_1\times\PP_2$-generic filter and $G=G_1\times G_2$, then
$\V=\V[G_1]\cap \V[G_2]$.
\end{lemma}

\begin{proof}   This is  Corollary~2.3 of \cite{kaplan}.
\end{proof}

\begin{lemma}  \label{product3}
Let $\theta(x)$ be any formula of set theory, possibly with hidden parameters from $\b V$, and let $\V[G]$ be any forcing extension of $\b V$.
Suppose that there is some $b\in \V[G]$ such that for every forcing extension $\V[G][H]$ of $\V[G]$,
\[\V[G][H]\models \theta(b)\wedge \exists^{=1}x\theta(x)\]
Then $b\in \b V$.
\end{lemma}

\begin{proof}  Fix $\theta(x), \V[G]$ and $b$ as above.  Let $\PP\in\b V$ be the forcing notion for which $G$ is $\PP$-generic. Let $\tau$ be a $\mathbb{P}$-name such that $b = \val(\tau, G)$.
Choose $p \in G$ such that $$p \forces \hbox{``$\mbox{for all forcing notions }\QQ, \forces_{\QQ} \theta(\check{\tau}) \land \exists^{=1} x \theta(x)$''}$$
Let $H$ be $\mathbb{P}$-generic over $\mathbb{V}[G]$ with $p \in H$. So $G \times H$ is $\mathbb{P} \times \mathbb{P}$-generic over $\mathbb{V}$. Let $i_1, i_2: \mathbb{P} \to \mathbb{P} \times \mathbb{P}$ be the canonical injections.
Then, since $(p, p) \in G \times H$, we have that 
$$\mathbb{V}[G][H] \models \theta(\val(i_1(\tau), G \times H)) \land \theta(\val(i_2(\tau)), G \times H) \land \exists^{=1} x \theta(x).$$ Hence $\mathbb{V}[G][H] \models \val(i_1(\tau), G \times H) = \val(i_2(\tau), G \times H)$ and so by Lemma~\ref{product2}, $\val(i_1(\tau), G \times H) \in \mathbb{V}$. But $\val(i_1(\tau), G \times H) = b$ so we are done.
\end{proof}


Lemma~\ref{product1} lends credence to the following definition.

\begin{definition}  Suppose that $\phi(x)$ is $\HC$-forcing invariant.  Then $\phi_\ext$ is the class of all sets $A$
such that $A\in \b V$ and, for some (equivalently, for every) forcing extension $\V[G]$ of $\b V$ with $A\in \HC^{\V[G]}$, we have $A \in \phi(\HC)^{\mathbb{V}[G]}$.
\end{definition}

As motivation for the notation used in the definition above, $\phi_\ext$ describes the class of all $A\in\V$ that are {\em potentially} in $\phi(\HC)$.
We are specifically interested in those $\HC$-forcing invariant $\phi$ for which $\phi_\ext$ is a {\bf set} as opposed to a proper class.  
%

\begin{definition}  An $\HC$-forcing invariant formula $\phi(x)$ is {\em short} if $\phi_\ext$ is a set.
\end{definition}

We begin with some easy observations.  As notation, if $C$ is a  subclass of $\b V$, then define $\P(C)$ to be all sets $A$ in $\b V$ such that every element of $A$ is in $C$.
(This definition is only novel when $C$ is a proper class.)
Similarly, $P_{\aleph_1}(C)$ is the class of all sets $A\in\P(C)$    that are countable (in $\b V$!).
Let $\delta(x)$ be the formula \[\delta(x):=\exists h[h:x\rightarrow \omega \ \hbox{is 1-1]}\]
Given a formula $\phi(x)$, let $\P(\phi)(y)$ denote the formula $\forall x(x\in y\rightarrow \phi(x))$
and let $\P_{\aleph_1}(\phi)$ denote $\P(\phi)(y)\wedge\delta(y)$.

\begin{lemma}  \label{easyex}
\begin{enumerate}
\item  The $\ext$-operator commutes with boolean combinations, i.e., if $\phi$ and $\psi$ are both $\HC$-forcing invariant, then $(\phi\wedge\psi)_\ext=\phi_\ext\cap\psi_\ext$
and $(\neg\phi)_\ext=\b V\setminus \phi_\ext$.
\item  $\delta(\HC) = \HC$. In particular $\delta(x)$ is $\HC$-forcing invariant and $\delta_{\ptl} = \mathbb{V}$.
\item  If $\phi$ is $\HC$-forcing invariant, then so are both $\P(\phi)$ and $\P_{\aleph_1}(\phi)$.  Moreover, $\mathcal{P}(\phi)(\HC) = \mathcal{P}_{\aleph_1}(\phi)(\HC)$ and
$\P(\phi_\ext)=(\P(\phi))_\ext=(\P_{\aleph_1}(\phi))_\ext$.
\item  Suppose $s:\omega\rightarrow \HC$ is any map such that for each $n$, $\phi(x,s(n))$ is $\HC$-forcing invariant.  (Recall that $\HC$-forcing invariant formulas are permitted to have a parameter from $\HC$.)  Then $\psi(x):=\exists n (n\in\omega\wedge\phi(x,s(n)))$ is $\HC$-forcing invariant and $\psi_\ext=\bigcup_{n\in\omega}\phi(x,s(n))_\ext$.
\end{enumerate}
\end{lemma}

\begin{proof}  The verification of (1) and (2) is immediate, simply by unpacking definitions.

(3)  That $\P(\phi)(y)$ and $\P_{\aleph_1}(\phi)(y)$ are $\HC$-forcing invariant is routine.  Since $\delta(\HC) = \HC$ we have $\mathcal{P}(\phi)(\HC) = \mathcal{P}_{\aleph_1}(\phi)(\HC)$ and so $(\P(\phi))_\ext=(\P_{\aleph_1}(\phi))_\ext$. 

The only other thing to check is that $\mathcal{P}(\phi_{\ptl}) = (\mathcal{P}(\phi))_{\ptl}$. Begin by choosing any $A\in\P(\phi)_\ext$.  We must show that every element $a\in A$ is in $\phi_\ext$.  
Fix any element $a\in A$.  
Choose any forcing extension $\V[G]$ of $\b V$ with $A\in \HC^{\V[G]}$.  Then
\[\HC^{\V[G]}\models\forall x (x\in A\rightarrow \phi(x))\]
Since $A\in \HC^{\V[G]}$, $a\in \HC^{\V[G]}$ as well.  Thus, $\HC^{\V[G]}\models\phi(a)$, so $a\in\phi_\ext$.

Conversely, suppose $A\in\P(\phi_\ext)$.  This means that $A\in\b V$ and every element of $A$ is in $\phi_\ext$.
Choose any forcing extension $\V[G]$ of $\b V$ such that $A\in \HC^{\V[G]}$.  As $\HC^{\V[G]}$ is transitive, every element $a\in A$ is also
an element of $\HC^{\V[G]}$.  Thus, $\HC^{\V[G]}\models\phi(a)$ for every $a\in A$.  That is, $\HC^{\V[G]}\models\forall x(x\in A\rightarrow \phi(x))$, so $A\in(\P(\phi))_\ext$.

(4)  Note that $s \in \HC$ so can be used as a hidden parameter for $\psi(x)$. Choose $A\in  \HC$.  It is immediate from the definition of $\psi$ that $\HC\models\psi(A)$ if and only if $\HC\models\phi(A,s_n)$ for some $n\in\omega$.
As this equivalence relativizes to any $A\in \HC^{\V[G]}$, both statements follow.
\end{proof}

\subsection{Strongly definable families}

In the previous subsection, we described a restricted vocabulary of formulas.  
Here, we discuss parameterized  families of classes of $\HC^{\V[G]}$ that are describable in this vocabulary.

\begin{definition}
	Let $\phi$ be $\HC$-forcing invariant. A family $X=(X^{\b V[G]}: \b V[G] \mbox{ a forcing extension of } \b V)$ is \emph{strongly definable via $\phi$} if $X^{\b V[G]}=\phi(\HC)^{\b V[G]}$ always.  We say the family $X$ is a \emph{strongly definable family}, or just \emph{strongly definable}, if it is strongly definable via some $\HC$-forcing invariant  formula $\phi$.
\end{definition}
	
	We say that two $\HC$-forcing invariant formulas $\phi$ and $\psi$ are \emph{persistently equivalent} if $\phi(\HC)^{\b V[G]}=\psi(\HC)^{\b V[G]}$ for every forcing extension $\b V[G]$.  Persistently equivalent formulas give rise to the same strongly definable family, and if $X$ is strongly definable via  both $\phi$ and $\psi$, then $\phi$ and $\psi$ are persistently equivalent.  
    
    We note that the strongly definable families can also be defined as the HC-forcing invariant formulas modulo persistent equivalence; using this one can verify that all the results of this paper are really theorems of ZFC. 
	
\begin{definition}
	If $X$ is strongly definable, define $X_\ext$ to be the class of all sets $A\in\V$ such that $A\in X^{\V[G]}$ for some
	forcing extension $\V[G]$ of $\V$.   We call $X$ {\em short} if $X_\ext$ is a set as opposed to a proper class.
\end{definition}

Note that if $X$ is strongly definable via $\phi$, then $X_\ptl = \phi_\ptl$ and $X$ is short if and only if $\phi$ is short.



We can define operations on the collection of strongly definable families.  Of particular interest is the countable power set.  That is, in the notation preceding Lemma~\ref{easyex}, given any $\HC$-forcing invariant $\phi(x)$, the two $\HC$-invariant formulas $\P(\phi)$ and $\P_{\aleph_1}(\phi)$ are persistently equal.  Thus, if $X$ is strongly definable via $\phi$, then $\P(\phi)$ and $\P_{\aleph_1}(\phi)$ give rise to the same family, which we denote by $\P_{\aleph_1}(X)$.

We begin by enumerating several examples and easy observations that help establish our notation.

\begin{example}  \label{cases}
\begin{enumerate}
\item  $\omega$ is strongly definable via the $\HC$-forcing invariant formula $\phi_\omega(x):=$``$x$ is a natural number.''  Here,  $\omega_\ext=\omega$.
In particular, $\omega$ is short.
\item  $\omega_1$ is strongly definable via the $\HC$-forcing invariant formula ``$x$ is an ordinal.''  Here, $(\omega_1)_\ext=ON$, the class of all ordinals. Thus, $\omega_1$ is not short.\footnote{The symbol $\omega_1$, depending on context, refers to either the least uncountable ordinal $(\omega_1)^{\mathbb{V}}$, or else the parametrized family $((\omega_1)^{\mathbb{V}[G]}: \mathbb{V}[G]$ a forcing extension of $\mathbb{V})$. It is in the latter sense that it makes sense to say that `$\omega_1$ is strongly definable'. This kind of ambiguity will not create problems, since in practice, given a subset $X \subseteq HC$, there is only one natural definition of $X$. Thus there is only one natural way of viewing $X$ as a family parametrized by forcing extensions. 
}


\item  The set of reals, $\RR=\P_{\aleph_1}(\omega)$, is strongly definable via $\phi_\RR(x):=\P(\phi_\omega)$.  By Lemma~\ref{easyex}(3)
$\RR_\ext=\P(\omega) = \RR$, hence $\RR$ is short.

\item  $\mathcal{P}_{\aleph_1}(\RR)$, the set of countable sets of reals, is strongly definable either by $ \P_{\aleph_1}(\phi_\RR)$  or by $\P(\phi_\RR)$.
Via either definition, by Lemma~\ref{easyex}(3), $(\mathcal{P}_{\aleph_1}(\RR))_\ext=\P(\P(\omega))$, hence is short.
\item  More generally, if $X$ is short, then it follows from Lemma~\ref{easyex}(3) that $\P_{\aleph_1}(X)_{\ptl} = \P(X_{\ptl})$ and so $\mathcal{P}_{\aleph_1}(X)$ is short.
\item  For any $\alpha<\omega_1$, let $\HC_\alpha$ denote the sets in $\HC$ whose rank  is less than $\alpha$. Then each $\HC_\alpha$ is strongly definable (since the formula ``rank$(x) < \alpha$" is $\HC$-forcing invariant). Also, each $(\HC_\alpha)_\ext = \mathbb{V}_\alpha$, so each $\HC_\alpha$ is short.

\end{enumerate}
\end{example}




\begin{notation} 
Suppose that $X_1,\dots,X_n$ are each strongly definable families. We say {\em $\psi(X_1,\dots,X_n)$ holds persistently} if, for every forcing extension $\V[G]$, we have
\[\V[G]\models  \psi({X_1}^{\V[G]},\dots,{X_n}^{\V[G]}).\]
\end{notation}

We list three examples of this usage in the definition below.

\begin{definition}  Suppose that $f$, $X$ and $Y$ are each strongly definable families.
\begin{itemize}
\item  The notation $f:X\rightarrow Y$ {\em persistently} means that  $f^{\V[G]}: X^{\V[G]} \to Y^{\V[G]}$ for all forcing extensions $\V[G]$ of $\b V$.
\item  The notation $f:X\rightarrow Y$ {\em is persistently injective} means that $f:X\rightarrow Y$ persistently and additionally, for all forcing extensions
$\V[G]$ of $\b V$,  $f^{\V[G]}:X^{\V[G]}\rightarrow Y^{\V[G]}$ is 1-1.  
\item  The notion $f:X\rightarrow Y$ {\em is persistently bijective} means $f^{-1}:X\to Y$ is strongly definable and both $f:X\rightarrow Y$ and $f^{-1}:Y\rightarrow X$ are persistently injective.
\end{itemize}
\end{definition}
The reader is cautioned that when $f:X\rightarrow Y$ persistently (or is persistently injective), the image of $f$ need not be strongly definable.  Indeed the ``image'' of a strongly definable function is not well-behaved in many respects, including the lack of a surjectivity statement in the following proposition.

\begin{prop}  \label{mapfunction}
Suppose that $f$, $X$, and $Y$ are each strongly definable.
\begin{enumerate}
\item  Suppose $f:X\rightarrow Y$ persistently.  Then $f_\ext:X_\ext\rightarrow Y_\ext$, i.e., $f_\ext$ is a class function with domain $X_\ext$
and image contained in $Y_\ext$.
\item  If $f:X\rightarrow Y$ is persistently injective, then $f_\ext:X_\ext\rightarrow Y_\ext$ is injective as well.
\item  If $f:X\rightarrow Y$ is persistently bijective, then $f_\ext:X_\ext\rightarrow Y_\ext$ is bijective.
\end{enumerate}
\end{prop}

\begin{proof}  (1)  It is obvious that $f_\ext$ is a class of ordered pairs and is single-valued.  As well, by definition of $f_{\ext}$, if $(a,b)\in f_\ext$, then for any forcing extension $\V[G]$ of $\b V$ with
$(a,b)\in \HC^{\V[G]}$ we have $(a,b)\in f^{\V[G]}$.  Thus $\dom(f_\ext)\subseteq X_\ext$ and $\textrm{im}(f_\ext)\subseteq Y_\ext$.
To see that $\dom(f_\ext)$ is equal to $X_\ext$ requires Lemma~\ref{product3}.  Choose any $a\in X_\ext$ and look at the formula $\theta(a,y)^{\HC}$, where $\theta(x,y)$ defines $f$,
i.e., $f$ is strongly definable via $\theta$.
Let $\V[G]$ be any forcing extension of $\b V$ in which $a\in X^{\V[G]}$ (in particular $a \in \HC^{\mathbb{V}[G]}$).  Let $b:=f^{\V[G]}(a)$.
The definition of persistence tells us that the hypotheses of Lemma~\ref{product3} apply,
hence $b\in \b V$.  Thus, $(a,b)\in f_\ext$ and $a\in\dom(f_\ext)$.


(2)  Choose $a,b,c\in \b V$ such that $(a,c),(b,c)\in f_\ext$.  Choose a forcing extension $\V[G]$ of $\b V$ such that $a,b,c\in \HC^{\V[G]}$.
Thus, $(a,c),(b,c)\in f^{\V[G]}$.  As $f$ is persistently injective, it follows that $a=b$ holds in $\V[G]$ and hence in $\b V$.  So $f_\ext$ is injective.

(3)  From (2) we get that both $f_\ext$ and $(f^{-1})_\ext$ are  injective.  It follows that $f_\ext$ is bijective.
\end{proof}

We close this subsection with a characterization of surjectivity.  Its proof is simply an unpacking of the definitions.  However, 
$f:X\rightarrow Y$ being persistently surjective need not imply that the induced map $f_\ext:X_\ext\rightarrow Y_\ext$ is surjective.


\begin{lemma}  \label{surj}
Suppose that $f$, $X$, and $Y$ are each strongly definable via the $\HC$-forcing invariant formulas $\theta(x,y)$, $\phi(x)$, and $\gamma(y)$, respectively
and that $f:X\rightarrow Y$ persistently.  Then $f:X\rightarrow Y$ is persistently surjective if and only if the formula $\rho(y):=\exists x \theta(x,y)$ is
$\HC$-forcing invariant and persistently equivalent to $\gamma(y)$.
\end{lemma}

\subsection{Potential Cardinality}

%

\begin{definition}  \label{HCcard}
Suppose $X$ and $Y$ are strongly definable.  We say that {\em $X$ is $\HC$-reducible to $Y$}, written
$X\hcleq Y$, if there is a strongly definable $f$ such that $f:X\rightarrow Y$ is persistently injective.
As notation, we write $X\hclt Y$ if $X\hcleq Y$ but $Y\not\hcleq X$.  We also write $X\hceq Y$ if $X\hcleq Y$ and $Y\hcleq X$; this is apparently weaker than $X$ and $Y$ being in persistent bijection.
\end{definition}

The following notion will be very useful for our applications, as it can often be computed directly. With this and Proposition~\ref{shorty} we can prove otherwise difficult non-embeddability results for $\leq_{HC}$.


\begin{definition}
	Suppose $X$ is strongly definable.  The \emph{potential cardinality} of $X$, denoted $\|X\|$, refers to $|X_\ext|$ if $X$ is short, or $\infty$ otherwise.  By convention we say $\kappa<\infty$ for any cardinal $\kappa$.
\end{definition}


%
%

\begin{prop}  \label{shorty} Suppose $X$ and $Y$ are both strongly definable.
\begin{enumerate}
\item  If $Y$ is short and $X\leq_\HC Y$, then $X$ is short.
\item  If $X\hcleq Y$, then  $\|X\|\leq \|Y\|$.
\item  If $X$ is short, then $X\hclt \m P_{\aleph_1}(X)$.
\end{enumerate}
\end{prop}


\begin{proof}  (1) follows immediately from (2).

(2)  Choose a strongly definable $f$ such that $f:X\rightarrow Y$ is persistently injective. Then by Proposition~\ref{mapfunction}(2), $f_\ext:X_\ext\rightarrow Y_\ext$ is an injective class function.
Thus, $|X_\ext|\le|Y_\ext|$.

(3)  Note that for any strongly definable $X$,  $X\hcleq \m P_{\aleph_1}(X)$  is witnessed by the strongly definable map $x\mapsto \{x\}$.
For the other direction, suppose by way of contradiction that $X$ is short, but $\P_{\aleph_1}(X)\hcleq X$.
Also,  $(\P_{\aleph_1}(X))_\ext=\P(X_\ext)$ by Lemma~\ref{easyex}(3).
Thus, by (2), we would have that $|\P(X_\ext)|\le |X_\ext|$, which contradicts Cantor's theorem since $X_\ext$ is a set.
\end{proof}

Using the fact that $(\HC_\beta)_\ext=\b V_\beta$, the following Corollary is immediate.

\begin{cor}  If $X$ is strongly definable and $\|X\|\leq\beth_\alpha$ for some $\alpha<\omega_1$, then
$\HC_{\omega+\alpha+1}\not\hcleq X$.
\end{cor}

\subsection{Quotients}

We begin with the obvious definition.  

\begin{definition}  A pair $(X,E)$ is a {\em strongly definable quotient} if both $X$ and $E$ are strongly definable
and persistently, $E$ is an equivalence relation on $X$.
\end{definition}

There is an immediate way to define a reduction of two quotients:

\begin{definition}\label{quotientReductionDef}
	Let $(X,E)$ and $(Y,F)$ be strongly definable quotients.  Say $(X,E)\hcleq (Y,F)$ if there is a strongly definable $f$ such that all of the following hold persistently:
	\begin{itemize}
		\item $f$ is a subclass of $X\times Y$.
		\item The $E$-saturation of $\dom(f)$ is $X$.  That is, for every $x\in X$, there is an $x'\in X$ and $y'\in Y$ where $xEx'$ holds and $(x',y')\in f$.
		\item $f$ induces a well-defined injection on equivalence classes.  That is, if $(x,y)$ and $(x',y')$ are in $f$, then $xEx'$ holds if and only if $yFy'$ does.
	\end{itemize}
	
	Define $(X,E)\hclt(Y,F)$ and $(X,E)\hceq (Y,F)$ in the natural way.
\end{definition}

We wish to define the potential cardinality $||(X, E)||$. It turns out that $|X_{\ptl}/E_{\ptl}|$ is too small, typically. For our purposes, we can restrict to more well-behaved quotients.


\begin{definition}  \label{rep}
A {\em representation} of a strongly definable quotient $(X,E)$ is  a  pair $f,Z$ of strongly definable families
such that 
$f:X\rightarrow Z$ is persistently surjective and persistently,
\[\forall a,b\in X [ E(a,b)\Leftrightarrow f(a)=f(b) ]\]
We say that $(X,E)$ is {\em representable} if it has a representation.
\end{definition}

In the case that $(X,E)$ is representable, the set of $E$-classes is strongly definable in a sense -- we equate it with the representation. For this reason we will also say $Z$ is a representation of $(X,E)$.  Note that if $f_1:X\rightarrow Z_1$ and $f_2:X\rightarrow Z_2$ are two representations of $(X,E)$, then there is a persistently bijective,
strongly definable  $h:Z_1\rightarrow Z_2$.  This observation implies the following definition is well-defined.


\begin{definition}  \label{quotdef}
If $(X,E)$ is a representable strongly definable quotient, then define $||(X,E)||=||Z||$ for some (equivalently, for all)
$Z$ such that there is a representation $f:(X,E)\to Z$.
\end{definition}

The following lemma can be proved by a routine composition of maps.

\begin{lemma}  \label{compo}  Suppose $(X,E)$ and $(Y,E')$ are strongly definable quotients, with representations $f:(X,E)\rightarrow Z$ and $g:(Y,E')\rightarrow Z'$.
Suppose $h$ is a witness to $(X,E)\hcleq (Y,E')$.  Then the induced function $h^*:Z\rightarrow Z'$ is  strongly definable, persistently injective, and  witnesses  $Z\hcleq Z'$.
\end{lemma}

We close this section with an observation about restrictions of representations.

\begin{lemma}  \label{restriction}  Suppose $f: (X,E)\rightarrow  Z$ is a representation and $Y \subseteq X$ is strongly definable and persistently  $E$-saturated.  Let 
$E'$ and $g=f\mr{Y}$ be the restrictions of $E$ and $f$, respectively, to $Y$.  Then the image $g(Y)$  is strongly definable, and so $g: (Y,E')\rightarrow g(Y)$ is a representation.
\end{lemma}

\begin{proof} We show that $\phi(z):=\exists y (y\in Y\wedge g(y)=z)$ is $\HC$-forcing invariant.  Fix any $z\in \HC$ and let $\mathbb{V}[G]$ be a forcing extension of $\mathbb{V}$. We show that $z \in \phi(\HC)$ iff $z \in \phi(\HC)^{\mathbb{V}[G]}$. Left to right is clear. For right to left: choose a witness $y\in Y^{\V[G]}$ such that $g(y)=z$ in $\V[G]$.
As $f(y)=z$, we conclude $z\in Z^{\V[G]}$ and hence $z\in Z^{\V}$ as $Z$ is strongly definable.  So, choose $y^*\in X^{\V}$ with $f(y^*)=z$.
Thus, in $\V[G]$, $E(y,y^*)$ holds.  As $Y$ is persistently $E$-saturated, $y^*\in Y^{\V[G]}$.  Since $Y$ is strongly definable, we conclude $y^*\in Y^{\V}$,
so $y^*$ witnesses that $z\in (g(Y))^{\V}$.  

As this argument relativizes to any forcing extension, we conclude that $\phi(z)$ is $\HC$-forcing invariant.
\end{proof}

All of the examples we work with will be representable, where the representations are Scott sentences.  Therefore this simple definition of $\|(X,E)\|$ will suffice completely for our purposes.  In the absence of a representation, one can still define $\|(X,E)\|$ using the notion of \emph{pins}; see for instance \cite{ZapletalForcingBorel} for a thorough discussion.



\section{Connecting Potential Cardinality with Borel Reducibility}\label{cssSection}

The standard framework for Borel reducibility of invariant classes is the following.
Let $L$ be a countable langauge and let $X_L$ be the set of $L$-structures with universe $\omega$. Endow $X_L$ with the usual logic topology; then $X_L$ becomes a Polish space. Moreover, if $\Phi$ is a sentence of $L_{\omega_1 \omega}$ then $\Mod(\Phi)$ is a Borel subset of $X_L$; hence $\Mod(\Phi)$ is a standard Borel space. 
The relation $\iso_\Phi$ is the restriction of the isomorphism relation to $\Mod(\Phi)\times\Mod(\Phi)$.  When no ambiguity arises we omit the $\Phi$.
If $L'$ is another countable language and $\Phi'$ is a a sentence of $L'_{\omega_1 \omega}$, then a Borel reduction from $(\Mod(\Phi), \cong) \to (\Mod(\Phi'), \cong)$ is a Borel map $f: \Mod(\Phi) \to \Mod(\Phi')$ such that, for all $M, N \in \Mod(\Phi)$, $M \cong N$ if and only if $f(M) \cong f(N)$.

We want to apply the machinery of the previous section to this setup. First, recall that we are working entirely in $ZFC$; thus a language $L$ is just a set with an arity function, and an $L$-structure with universe $\omega$ is just a function $f: L \to \bigcup_n \mathcal{P}(\omega^n)$ respecting the arities. Since our languages are countable we can suppose that they are elements of $HC$. We will presently show that for any sentence $\Phi$ of $L_{\omega_1 \omega}$, $(\Mod(\Phi), \cong)$ is a strongly definable quotient. We will also show that $(\Mod(\Phi), \cong)$ is representable, and that Borel reductions are in particular $\HC$-reductions.

\subsection{Canonical Scott sentences}

For what follows,
we need the notion of a canonical Scott sentence of any infinite $L$-structure, regardless of cardinality.  The definition below is  in both Barwise~\cite{BarwiseScottSentences}
and Marker~\cite{MarkerMT}. 


\begin{definition} \label{cssDefinition} Suppose $L$ is countable and $M$ is any  infinite $L$-structure, say of power $\kappa$.
For each $\alpha<\kappa^+$, define an $L_{\kappa^+,\omega}$ formula $\phi_\alpha^\abar(\xbar)$ for each finite $\abar\in M^{<\omega}$ as follows:
\begin{itemize}
\item  $\phi_0^\abar(\xbar):=\bigwedge \{\theta(\xbar): \theta$ atomic or negated atomic and $M\models\theta(\abar)\}$;
\item  $\phi_{\alpha+1}^\abar(\xbar):=\phi_\alpha^\abar(\xbar)\ \wedge\  \bigwedge \left\{ \exists y\, \phi_\alpha^{\abar, b}(\xbar,y):b\in M\right\}\ \wedge\forall y\bigvee \left\{\phi_\alpha^{\abar, b}(\xbar,y):b\in M\right\}$;
\item  For $\alpha$ a non-zero limit, $\phi_\alpha^\abar(\xbar):=\bigwedge \left\{\phi_\beta^\abar(\xbar):\beta<\alpha\right\}$.
\end{itemize}
Next, let $\alpha^*(M)<\kappa^+$ be least  ordinal $\alpha$ such that 
for all finite $\abar$ from $M$, 
\[\forall \xbar [\phi_\alpha^\abar(\xbar)\rightarrow\phi_{\alpha+1}^{\abar}(\xbar)].\]

Finally, put 
$\css(M):=\phi^{\emptyset}_{\alpha^*(M)} \wedge \bigwedge \left\{ \forall\, \xbar [\phi^{\abar}_{\alpha^*(M)}(\xbar)\rightarrow \phi^{\abar}_{\alpha^*(M)+1}(\xbar)]:\abar\in M^{<\omega}\right\}$.
\end{definition}


For what follows, it is crucial that the choice of $\css(M)$ really is canonical.  In particular, in the infinitary clauses forming the definition of $\tp_{\alpha+1}^\abar(\xbar)$,  we consider the conjunctions and disjunctions be taken over {\em sets} of formulas, as opposed to sequences.  In particular, we ignore the multiplicity of a formula inside the set.
By our conventions about working wholly in ZFC, countable languages and sentences of $L_{\infty,\omega}$ are sets, and in particular
$\css(M)$ is a set. 

We summarize the well-known, classical facts about canonical Scott sentences with the following:

\begin{fact}  \label{summ}  Fix a countable language $L$.
\begin{enumerate}
\item  For every $L$-structure $M$, $M\models \css(M)$; and for all 
 $L$-structures $N$, $M\equiv_{\infty,\omega} N$ if and only if $\css(M)=\css(N)$ if and only if $N\models \css(M)$.
 \item  If $M$ is countable, then $\css(M)\in \HC$.
\item  $\css$ is absolute between transitive models of $ZFC^-$, where $ZFC^-$ is $ZFC$ but with the powerset axiom is deleted. (Recall that $HC \models ZFC^-$.)
\item If $M$ and $N$ are both countable, then $M\iso N$ if and only if $\css(M)=\css(N)$ if and only if $N\models \css(M)$.
\end{enumerate}
\end{fact}

Our primary interest in canonical Scott sentences is that they give rise to representations of classes of $L$-structures.  The key to the representability is Karp's Completeness Theorem for sentences of $L_{\omega_1,\omega}$, see e.g., Theorem~3 of Keisler~\cite{Keisler}, which says that if a sentence $\sigma$ of $L_{\omega_1 \omega}$ is consistent, then it has a countable model. It quickly follows that if $\sigma$ is a sentence of $L_{\omega_1 \omega}$, and $\sigma$ has a model in a forcing extension, then $\sigma$ already has a countable model in $\mathbb{V}$.  

We begin by considering $\CSS(L)$, the set of all canonical Scott sentences of  structures in $X_L$, the set of $L$-structures with universe $\omega$.

\begin{lemma} \label{Karp0}
Fix a countable language $L$.
Then:
\begin{enumerate}
\item  $\CSS(L)$ is strongly definable via the formula $\phi(y):=\exists M(M\in X_L\wedge \css(M)=y)$;
\item  The strongly definable function $\css:X_L\rightarrow \CSS(L)$ is persistently surjective;
\item  $\css:X_L\rightarrow \CSS(L)$ is a representation of the
strongly definable quotient $(X_L,\iso)$.
 \end{enumerate}
 \end{lemma}
 
 \begin{proof} Note that $L_{\omega_1 \omega}$ is strongly definable, clearly.
 
 (1) We need to verify that $\phi(y)$ is $\HC$-forcing invariant.  Suppose $\sigma\in \HC$ and $\V[G]$ is a forcing extension of $\b V$. If $\sigma\in \phi(\HC)^{\V[G]}$, then there is some $M\in X_L^{\V[G]}$ such that $\css(M)=\sigma$. Hence by the preceding discussion, $\sigma$ has a countable model $N \in (X_L)^{\b V}$.  In $\V[G]$, $N\models\sigma$; but $\sigma = \css(M)$.  So $\css(N)=\sigma$.
 As this argument readily relativizes to any forcing extension, 
 $\phi$ is $\HC$-forcing invariant.
 
(2) follows  from (1) and Lemma~\ref{surj}.
 
 (3): $(X_L, \iso)$ is a strongly definable quotient since $\css$ is strongly definable. Thus we conclude by (2) and Fact~\ref{summ}(4).
  \end{proof}

In most of our applications, we are interested in strongly definable subclasses of $X_L$ that are closed under isomorphism.  
For any sentence $\Phi$ of $L_{\omega_1,\omega}$,  because $\Mod(\Phi)$ is a Borel subset of $X_L$,  it follows from Shoenfield's Absoluteness Theorem that both $\Mod(\Phi)$  
and the restriction of $\css$   to $\Mod(\Phi)$ (also denoted by $\css$) $\css:\Mod(\Phi)\rightarrow \HC$ are strongly definable.

\begin{definition}  For $\Phi$ any sentence of $L_{\omega_1,\omega}$, $\CSS(\Phi)=\{\css(M):M\in\Mod(\Phi)\}\subseteq HC$.
\end{definition}

\begin{prop} \label{Karp1}   Fix any sentence $\Phi\in L_{\omega_1,\omega}$ in a countable vocabulary.
Then $\css:\Mod(\Phi)\rightarrow\CSS(\Phi)$ is a representation of the quotient $(\Mod(\Phi),\iso)$. In particular the latter is strongly definable.
\end{prop}

\begin{proof}  As $\Mod(\Phi)$ is strongly definable, this follows immediately from Lemmas~\ref{Karp0} and \ref{restriction}.
\end{proof}

 \begin{definition}  \label{shortdef}
 Let $\Phi$ be any sentence of $L_{\omega_1,\omega}$ in a countable vocabulary.  We say that $\Phi$ is {\em short} if
 $\CSS(\Phi)_\ext$ is a set (as opposed to a proper class).
If $\Phi$ is short, let the \emph{potential cardinality} of $\Phi$, denoted $||\Phi||$, be the (usual) cardinality of $\CSS(\Phi)_\ext$; otherwise let it be $\infty$.  
\end{definition}

It follows from Proposition~\ref{Karp1} and Definitions~\ref{HCcard} and \ref{quotdef} that
\[||\Phi||=||(\Mod(\Phi),\iso)||=||\CSS(\Phi)||=|\CSS(\Phi)_\ext|.\]

\medskip

In order to understand the class $\CSS(\Phi)_\ext$, note that if $\phi\in \CSS(\Phi)_\ext$, then $\phi\in\V$ and  is  a sentence of $L_{\infty,\omega}$.
If we choose any forcing extension $\V[G]$ of $\b V$ for which $\phi\in \HC^{\V[G]}$, then $\V[G]\models \, \hbox{`$\phi\in L_{\omega_1,\omega}$'}$
and there is some $M\in \HC^{\V[G]}$ such that $V[G]\models \hbox{`$M\in \Mod(\Phi)\ \hbox{and}\ \css(M)=\phi$'}$.  Thus, we  refer to elements of $\CSS(\Phi)_\ext$
as being {\em potential canonical Scott sentences} of a model of $\Phi$.  In particular, every element of $\CSS(\Phi)_\ext$ is {\em potentially satisfiable} in the sense that it
is satisfiable in some forcing extension $\V[G]$ of $\V$.  
There is a proof system
for sentences of $L_{\infty,\omega}$ for which a sentence is consistent if and only if it is potentially satisfiable as defined above.\footnote{In Chapter 4 of \cite{Keisler}, Keisler gives a proof system for $L_{\omega_1 \omega}$, and shows in Theorem 3 that it is complete, i.e. if $\phi$ is unprovable then $\lnot \phi$ has a model. The natural generalization of this proof system to $L_{\infty \omega}$ works: the proof of Theorem 3 shows that whenever $\phi$ is unprovable, then $\lnot \phi$ lies in a consistency property. Forcing on the consistency property gives a model of $\lnot \phi$.} When we say `$\phi$ implies $\psi$', we mean with respect to this proof system; equivalently, in any forcing extension $\V[G]$, every model of $\phi$ is a model of $\psi$. 


%
%

We can also ask what is the image of the class function $\css$ when restricted to the class of models of $\Phi$.  As notation, let $\CSS(\Phi)_\sat$
denote the class $\{\css(M):M\in\V \textrm{ and } M\models\Phi\}$.  This choice of notation is clarified by the following easy lemma.

\begin{lemma} \label{subset} $\CSS(\Phi)_\sat\subseteq \CSS(\Phi)_\ext$.
\end{lemma}

\begin{proof}  Choose any $\phi \in\CSS(\Phi)_\sat$ and choose any $M\in\V$ such that $M\models\Phi$ and $\css(M)=\phi$.  Then clearly, $\phi\in\V$.
As well, choose a forcing extension $\V[G]$ in which $M$ is countable.  Then, in $\V[G]$, there is some $M'\in\Mod(\Phi)$ (i.e., where the universe of $M'$ is $\omega$)
such that $M'\iso M$.  Then $(\css(M'))^{\V[G]}=\phi$ and so $\phi\in\CSS(\Phi)_\ext$.
\end{proof}

To summarize, elements of $\CSS(\Phi)_\ext$ are called {\em potential} canonical Scott sentences, whereas elements of $\CSS(\Phi)_\sat$ are {\em satisfiable}.
This suggests a property of the sentence $\Phi$.

\begin{definition} \label{grounded}
 A sentence $\Phi\in L_{\omega_1,\omega}$ (or a complete first-order theory $T$) is {\em grounded} if $\CSS(\Phi)_\sat=\CSS(\Phi)_\ext$, i.e., if every potential canonical Scott sentence
is satisfiable.
\end{definition}

As a trivial example, if $T$ is $\aleph_0$-categorical, then as all models of $T$ are back-and-forth equivalent, $\CSS(T)_\ext$ is a singleton, hence $T$ is grounded.
In Section~\ref{REFsection} we show that both of the theories $\REF(\bin)$ and $\REF(\inf)$ are grounded, but  in Section~\ref{OmegaStableSection} we prove that the theory $\TK$ is not grounded.

Next, we show that a  Borel reduction between invariant classes yields a strongly definable map between the associated canonical Scott sentences.

\begin{fact}  \label{corresp}
Suppose $\Phi$ and $\Phi'$ are sentences of $L_{\omega_1,\omega}$ and $L'_{\omega_1,\omega}$ respectively.
If there is a Borel reduction $f:(\Mod(\Phi),\iso)\rightarrow (\Mod(\Phi'),\iso)$ then there is a strongly definable  $f^*:\CSS(\Phi)\rightarrow \CSS(\Phi')$
between canonical Scott sentences that is persistently injective. Hence $\CSS(\Phi) \leq_{HC} \CSS(\Phi')$.
\end{fact}

\begin{proof}  It is a standard theorem, see e.g., \cite{KechrisDST} Proposition 12.4, that the graph of $f$ is Borel.  So $f$ is naturally a strongly definable family of Borel sets. 
By Lemma~\ref{compo}, it suffices to show that $f: (\Mod(\Phi),\iso) \leq_{\HC} (\Mod(\Phi'),\iso)$, which amounts to showing that $f$ remains well-defined and
injective on isomorphism classes in every forcing extension.  But this is a ${\bf \Pi^1_2}$ statement in codes for $f,\Phi,\Phi'$, and thus is absolute to forcing extensions by
Shoenfield's Absoluteness Theorem.  
\end{proof}


The following Theorem is simply a distillation of our previous results.
 

\begin{theorem}  \label{translate}
Let $\Phi$ and $\Psi$ be sentences of $L_{\omega_1,\omega}$, possibly in different countable vocabularies.
\begin{enumerate}
\item  If $\Psi$ is short, while $\Phi$ is not short, then $(\Mod(\Phi),\iso)$ is not Borel reducible to  $(\Mod(\Psi),\iso)$.
\item If $||\Psi||<||\Phi||$, then 
$(\Mod(\Phi),\iso)$ is not Borel reducible to $(\Mod(\Psi),\iso)$.
\end{enumerate}
\end{theorem}

\begin{proof}
(1) follows immediately from (2).

(2) Suppose $f:(\Mod(\Phi),\iso)\rightarrow (\Mod(\Psi),\iso)$ were a Borel reduction.  Then by Fact~\ref{corresp} we would obtain
a strongly definable $f^*:\CSS(\Phi)\rightarrow \CSS(\Psi)$ that is persistently injective, meaning
$||\CSS(\Phi)||\le||\CSS(\Psi)||$.  This contradicts Proposition~\ref{shorty}(2).
\end{proof}

\subsection{Consequences of $\iso_\Phi$ being Borel}

Although our primary interest is classes $\Mod(\Phi)$ where $\iso_\Phi$ is not Borel, in this brief subsection we see the consequences of $\iso_\Phi$ being Borel.

\begin{theorem} \label{charBorel} The following are equivalent for a sentence $\Phi\in L_{\omega_1,\omega}$ in countable vocabulary.
\begin{enumerate}
\item  The relation of $\iso$ on $\Mod(\Phi)$ is a Borel subset of $\Mod(\Phi)\times \Mod(\Phi)$;
\item For some $\alpha < \omega_1$, $\CSS(\Phi) \subseteq \HC_\alpha$;
\item  For some $\alpha<\omega_1$, $\CSS(\Phi)$ is persistently contained in $\HC_\alpha$;
\item  $\CSS(\Phi)_\ext$ is contained in $\b V_\alpha$ for some $\alpha<\omega_1$.
\end{enumerate}
\end{theorem}

\begin{proof}   
To see the equivalence of (1) and (2), first note that in both conditions we are only considering models of $\Phi$ with universe $\omega$
and the canonical Scott sentence of such objects.  In particular, neither condition involves passing to a forcing extension.
However, it is a classical result (see for instance \cite{GaoIDST}, Theorem 12.2.4) that $\iso$ is Borel if and only if the Scott ranks of countable models are bounded below $\omega_1$,
which is equivalent to stating that there is a bound on the canonical Scott sentences in the $\HC_\alpha$ hierarchy.

For (2) implies (3),   note that the formula $\exists M: M \models \Phi \land \css(M) \not \in \HC_\alpha$ is a $\Sigma_1$ formula in the parameters $\Phi, \alpha\in HC$ and so is absolute to forcing extensions by Lemma~\ref{sigma1}.

That (3) implies (4) and (4) implies (2)  follow directly from Example~\ref{cases}(6).
\end{proof}


We obtain an immediate corollary to this.  
Let  $I_{\infty,\omega}(\Phi)$ denote the cardinality of a maximal set of pairwise $\equiv_{\infty,\omega}$-inequivalent models 
$M\in\V$ (of any cardinality) with $M\models \Phi$.
If no maximal set exists, we write $I_{\infty,\omega}(\Phi)=\infty$.  By Fact~\ref{summ} and  Lemma~\ref{subset}, $I_{\infty,\omega}(\Phi)=|\CSS(\Phi)_\sat|\le ||\Phi||$.


\begin{cor}  \label{Borelisshort}
Let $\Phi$ be any sentence in $L_{\omega_1,\omega}$ in a countable vocabulary such that $\iso$ is a Borel subset of $\Mod(\Phi)\times \Mod(\Phi)$.
Then
\begin{enumerate}
\item   $\Phi$ is short; and 
\item $I_{\infty,\omega}(\Phi)<\beth_{\omega_1}$. (In fact $||\Phi|| < \beth_{\omega_1}$.)
\end{enumerate}
\end{cor}

\begin{proof}
Assume that $\iso$ is a Borel subset of $\Mod(\Phi)\times \Mod(\Phi)$.  By Theorem~\ref{charBorel}(3), $\CSS(\Phi)_\ext\subseteq \b V_{\alpha}$  for some
$\alpha<\omega_1$ and hence is a set.
Thus, $\Phi$ is short, and $I_{\infty,\omega}(\Phi)=|\CSS(\Phi)_\sat|\le |\CSS(\Phi)_\ext|\le|\V_\alpha|<\beth_{\omega_1}$.
\end{proof}

We remark that the implication in Corollary~\ref{Borelisshort} does not reverse.  In Sections~\ref{REFsection} and \ref{OmegaStableSection} we show that both of the complete theories $\REF(\bin)$ and $\K$ are short,  but on countable models of either theory, $\iso$ is not Borel. 



\subsection{Maximal Complexity}
In this subsection, we recall two definitions of maximality.  The first, Borel completeness, is from Friedman-Stanley \cite{FriedmanStanleyBC}.  

\begin{definition} \label{BC} Let $\Phi\in L_{\omega_1,\omega}$.  
	The quotient $(\Mod(\Phi),\iso)$ is {\em Borel complete} if every $(\Mod(\Psi),\iso)$ is Borel reducible to  $(\Mod(\Phi),\iso)$.
\end{definition}

\begin{cor}  \label{BCnotshort}  If $(\Mod(\Phi),\iso)$ is Borel complete, then $\Phi$ is not short.
\end{cor}

\begin{proof}  Let $L=\{\le\}$ and let $\Psi$ assert that $\le$ is a linear ordering.  As $(\Mod(\Phi),\iso)$ is Borel complete,
there is a Borel reduction $f:(\Mod(\Psi),\iso)\rightarrow (\Mod(\Phi),\iso)$.  However, it is easily proved that for distinct ordinals  $\alpha\neq\beta$,
the $L$-structures $(\alpha,\le)$ and $(\beta,\le)$ are $\equiv_{\infty\omega}$-inequivalent models of $\Psi$, hence have distinct canonical Scott sentences.
Thus, $\CSS(\Psi)_\sat$ is a proper class, and hence so is $\CSS(\Psi)_\ext$ by Lemma~\ref{subset}.  So $\Phi$ cannot be short by Theorem~\ref{translate}(1).
\end{proof}


If one is only interested in classes of countable models, then the Borel complete classes are clearly maximal with respect to Borel reducibility.  As any invariant class 
of countable structures has a natural extension to a class of uncountable structures,  one can ask for more.
The following definitions from \cite{LaskowskiShelahAleph0Stable} generalize Borel completeness to larger cardinals $\lambda$.  To see that it is a generalization, recall that among countable structures, isomorphism is equivalent to back-and-forth equivalence, and that for structures of size $\lambda$, $\equiv_{\lambda^+,\omega}$-equivalence is also equivalent to back-and-forth equivalence.  Consequently, `Borel complete' in the sense of Definition~\ref{BC} is equivalent to `$\aleph_0$-Borel complete'  in Definition~\ref{lambdaBC}.
So, `$\Phi$ is $\lambda$-Borel complete for all infinite $\lambda$' implies $\Phi$ Borel complete.  However, in Section~\ref{OmegaStableSection} 
we will see that the theory $\TK$ is Borel complete,
but is not $\lambda$-Borel complete for large $\lambda$.

\begin{definition}  \label{lambdaBC} Let $\Phi$ be a sentence of $L_{\omega_1, \omega}$.

	\begin{itemize}
		\item  For $\lambda\ge\aleph_0$, let $\Mod_\lambda(\Phi)$ denote the class of models of $\Phi$ with universe $\lambda$.  
		\item Toplogize $\Mod_\lambda(\Phi)$ by declaring that 
		${{\cal B}}:=\{U_{\theta(\alphabar)}:\theta(\xbar) \ \hbox{is quantifier free and $\alphabar\in\lambda^{<\omega}$}\}$
		is a sub-basis, 
		where $U_{\theta(\alphabar)}=\{M\in \Mod_\lambda(\Phi):M\models\theta(\alphabar)\}$.
		\item  A set is  {\em $\lambda$-Borel}  if it is in the $\lambda^+$-algebra generated by the sub-basis ${\cal B}$.
		\item  A function $f:\Mod_\lambda(\Phi)\rightarrow \Mod_\lambda(\Psi)$ is a
		{\em $\lambda$-Borel embedding} if
		\begin{itemize} 
			\item  the inverse image of every (sub)-basic open set is $\lambda$-Borel; and
			\item  For $M,N\in \Mod_\lambda(\Phi)$, $M\equiv_{\infty,\omega} N$ if and only if  $f(M)\equiv_{\infty,\omega} f(N)$.
		\end{itemize}
		\item $(\Mod_\lambda(\Phi),\equiv_{\infty,\omega})$ is {\em $\lambda$-Borel reducible} to $(\Mod_\lambda(\Psi),\equiv_{\infty,\omega})$ if there exists a $\lambda$-Borel embedding
		$f:\Mod_\lambda(\Phi)\rightarrow \Mod_\lambda(\Psi)$.
		\item  $\Phi$ is {\em $\lambda$-Borel complete }  if every $(\Mod_\lambda(\Psi),\equiv_{\infty,\omega})$ is $\lambda$-Borel reducible to $(\Mod_\lambda(\Phi),\equiv_{\infty,\omega})$.
	\end{itemize}
\end{definition}
For example, the class of graphs (directed or undirected) is $\lambda$-Borel complete for all infinite $\lambda$. This is a standard coding argument.  
Although we are not aware of any direct reference,  Theorem 5.5.1 of \cite{HodgesModelTheory} states  that graphs can interpret any theory. 
It is easily checked that the  map constructed  in the proof of Theorem~5.5.1 is in fact a 
$\lambda$-Borel reduction for every $\lambda$.  

Also, in \cite{LaskowskiShelahAleph0Stable} it is proved that the class of subtrees of $\lambda^{<\omega}$ is $\lambda$-Borel complete, and more recently the second author has proved that the class of linear orders is $\lambda$-Borel complete for all $\lambda$.

\subsection{Jumps and products}


In this subsection we recall two procedures -- the \emph{jump} and the \emph{product} -- and use them to define a sequence $\<T_\alpha:\alpha\in\omega_1\>$ of complete, first order theories
for which the potential cardinality is strictly increasing.

\begin{definition}\label{jumpDefinition}
	Suppose $L$ is a countable relational language and $\Phi\in L_{\omega_1,\omega}$.  The {\em jump of $\Phi$}, written $J(\Phi)$, is a sentence of $L'_{\omega_1 \omega}$ defined as follows, where $L' =L\cup\{E\}$ is obtained by adding a new binary relation symbol $E$ to $L$. Namely $J(\Phi)$ states that $E$ is an equivalence relation with infinitely many classes, each of which is a model of $\Phi$.  If $R\in L$ and $\o x$ is a tuple not all from the same $E$-class, then $R(\o x)$ is defined to be false, so that the models are independent.
\end{definition}

There is a corresponding notion of jump that can be defined directly on equivalence relations: Given an equivalence relation $E$ on $X$, its jump is the equivalence relation $J(E)$ on $X^\omega$, defined by setting $(x_n: n \in \omega) J(E) (y_n: n  \in \omega)$ if there is some $\sigma \in S_\infty$ with $x_{\sigma(n)} E y_n$ for all $n \in \omega$. Then the previous definition of the jump can be viewed as the special case where $(X, E)$  is $(\Mod(\Phi), \cong)$.

The notion of a jump was investigated in \cite{FriedmanStanleyBC}, where it was shown that if $E$ is a Borel equivalence relation on a Polish space $X$ with more than one class, then $E\borellt J(E)$. We give a partial generalization of this in Proposition~\ref{BasicJumpProposition}(3) -- if $\Phi\in L_{\omega_1,\omega}$ is short, then $\|\Phi\|<\|J(\Phi)\|$, so $\Phi\borellt J(\Phi)$.  Using the theory of pins \cite{ZapletalForcingBorel}, one can use essentially the same proof to give a true generalization: if $(X,E)$ is strongly definable, short, and has more than one $E$-class, then $\|(X,E)\|<\|(X^\omega,J(E))\|$, so in particular $E\borellt J(E)$.

The following Proposition lists the basic properties of the jump operation.

\begin{prop}\label{BasicJumpProposition}
	Let $\Phi$ and $\Psi$ be $L_{\omega_1,\omega}$-sentences in countable relational languages.
	
	\begin{enumerate}
		\item If $\Phi$ is a complete first order theory, so is $J(\Phi)$.
		\item If $\Phi$ is grounded, so is $J(\Phi)$.
		\item If $\Phi$ is short, then $J(\Phi)$ is also short. More precisely, if $\|\Phi\|$ is infinite, $\|J(\Phi)\|=2^{\|\Phi\|}$.  If $2\leq \|\Phi\|<\aleph_0$, $\|J(\Phi)\|=\aleph_0$.  If $\|\Phi\|=1$, then $\|J(\Phi)\|=1$.
		\item The jump is monotone: if $\Phi\borelleq \Psi$, then $J(\Phi)\borelleq J(\Psi)$.
		\item It is always true that $\Phi\borelleq J(\Phi)$; if $\Phi$ is short and not $\aleph_0$-categorical, then $\Phi\borellt J(\Phi)$.
	\end{enumerate}
	
	Note that here and throughout, we use $\Phi\borelleq \Psi$ as a shorthand for $(\Mod(\Phi),\iso)\borelleq (\Mod(\Psi),\iso)$, and similarly with $\borellt$ and $\boreleq$.
\end{prop}
\begin{proof}
	(1) That the jump is first-order is clear.  Completeness follows from a standard Ehrenfeucht-Fra\"isse argument.
	
	(2) Let $\Psi\in \CSS(J(\Phi))_\ptl$, and let $\b V[G]$ be some forcing extension in which $\Psi$ is hereditarily countable.  Let $M\models\Psi$ be the unique countable model of $\Psi$ in $\b V[G]$.  Let $X$ be the set of $E$-classes in $M$, and for each $x\in X$, let $\Psi_x$ be the canonical Scott sentence of $x$, viewed as an $L$-structure.  Let $m_x$ be the number of $E$-classes of $M$ which are isomorphic to $x$ as $L$-structures, if this number is finite; if infinite, let $m_x=\omega$.
	
	The set of pairs $S=\{(\Psi_x,m_x):x\in X\}$ depends only on the isomorphism type of $M$, so is uniquely definable from $\Psi$.  By Lemma~\ref{product3}, $S\in\b V$, although it may no longer be countable.  Since $\Phi$ is grounded, each $\Psi_x$ has a model in $\b V$, so let $N$ be the model coded from $S$ -- for each pair $(\Psi',m)$ in $S$, give $m$ distinct $E$-classes, each of which is a model of $\Psi'$.  Then $M\iso N$ in any sufficiently large forcing extension of $\b V[G]$, so $N\models\Psi$, as desired.
	
	(3) We assume $\|\Phi\|$ is infinite; the finite cases are similar and trivial, respectively.  First, let $X\subset \CSS(\Phi)_\ptl$ be arbitrary.  Let $\b V[G]$ be a forcing extension in which $X$ is hereditarily countable, and let $M_X$ be a countable model of $J(\Phi)$ such that each $E$-class of $M_X$ is a model of some $\Psi\in X$, and each $\Psi\in X$ is represented infinitely often.  Since each $\Psi$ is $\aleph_0$-categorical, $M_X$ is determined up to isomorphism by these constraints.  Therefore, $\Psi_X=\css(M_X)$ is determined entirely by $X$, so $\Psi_X\in \b V$ by Lemma~\ref{product3}.  Furthermore, if $X\not=Y$, then $M_X\not\iso M_Y$, so $\Psi_X\not=\Psi_Y$, so $\|J(\Phi)\|\geq 2^{\|\Phi\|}$.
	
	The other direction is similar to part (2). Let $\Psi\in \CSS(J(\Phi))_\ext$, let $\b V[G]$ be a forcing extension in which $\Psi$ is hereditarily countable, and let $M_\Psi\models \Psi$ be the unique countable model.  Let $X$ be the set of $E$-classes of $M_\Psi$, and for each $x \in X$, let $\Psi_x$ be $\css(x)$, where we consider the $E$-class $x$ as an $L$-structure and a model of $\Psi$.  For each $x$, let $m(x)$ be the number of equivalence classes of $M$ which are isomorphic to $x$ as $\mathcal{L}$-structures, or $\omega$ if there are infinitely many.  Since $M_\Psi$ is determined up to isomorphism by $\Psi$, the set $S_\Psi=\{(\Psi_x,m(x)): x \in X\}$ is determined entirely by $\Psi$.  Therefore, $S_\Psi\in \b V$ by Lemma~\ref{product3}.  If $\Psi\not=\Psi'$, $M_\Psi\not\iso M_{\Psi'}$, so $S_{\Psi}\not=S_{\Psi'}$.  Thus we see that $\|J(\Phi)\| \leq |(\omega+1)^{\|\Phi\|}|$.  Since $\|\Phi\|$ is infinite, this last is equal to $2^{\|\Phi\|}$, completing the proof.
	
	(4) follows from the fact that each equivalence class is in canonical bijection with $\omega$, allowing us to  apply $f$ to each class and reindex. (5) follows from (3) and the fact that if $\|\Phi\|>\|\Psi\|$ then $\Phi\not\borelleq \Psi$.
\end{proof}

Another important operation is the \emph{product}:

\begin{definition}
	Suppose $I$ is a countable set and for each $i$, $\Phi_i$ is a sentence of $L_{\omega_1,\omega}$ in the countable relational language $L_i$.  The \emph{product of the $\Phi_i$}, denoted $\prod_i \Phi_i$, is a sentence of $L_{\omega_1 \omega}$, where $L=\{U_i:i\in I\}\cup \bigcup_i L_i$ is the disjoint union of the $L_i$'s together with new unary predicates $\{U_i: i \in I\}$. 
    
    Namely $\prod_i \Phi_i$ states that the $U_i$ are disjoint, that the elements of $U_i$ form a model of $\Phi_i$ when viewed as an $L_i$-structure, and that if $R\in L_i$ and $\o x$ is not all from $U_i$, then $R(\o x)$ is false, so that the models are independent.  If $I$ is finite, we also require that each element is in some $U_i$.
    This becomes more convenient 
\end{definition}

The proofs of the corresponding facts for products are quite similar to those for the jump, so we omit them:

\begin{prop}\label{BasicProductProposition}
	Let $\{\Phi_i:i\in I\}$ and $\{\Psi_j:j\in J\}$ be countable sets of $L_{\omega_1,\omega}$-sentences in countable relational languages.
	
	\begin{enumerate}
		\item If each $\Phi_i$ is complete and first order, so is $\prod_i \Phi_i$.
		\item If each $\Phi_i$ is grounded, so is $\prod_i \Phi_i$.
		\item If each $\Phi_i$ is short then $\prod_i \Phi_i$ is short.  More precisely, $\|\prod_i \Phi_i\|= \kappa \cdot \prod_i \|\Phi_i\|$, where $\kappa$ is $1$ if $I$ is finite or $\aleph_0$ if $I$ is infinite.
		\item The product is monotone: if $f:I\to J$ is an injection and $\Phi_i\borelleq \Psi_{f(i)}$ for all $i\in I$, then $\prod_i \Phi_i\borelleq \prod_j \Psi_j$.
		\item It is always true that for all $i\in I$, $\Phi_i\borelleq \prod_i \Phi_i$.  If additionally, for all $i\in I$, there is an $i'\in I$ where $\Phi_i\borellt \Phi_{i'}$, then for all $i\in I$, $\Phi_i\borellt \prod_i \Phi_i$.
	\end{enumerate}
\end{prop}

Note that if we use $L_{\omega_1,\omega}$ to add to our definition of product that the $U_i$ are exhaustive, the product of first-order theories may not be first order, but the statement of (3) improves and we can lose reference to $\kappa$.  This can also be achieved by working in a multisorted first-order logic, which preserves (1) and fixes (3).  This is the approach we prefer, but for the time being we remain flexible.

We close this section by defining a set of concrete benchmarks: Note that these are essentially the same as the $\m I_\alpha$ in \cite{FriedmanStanleyBC}, the $\iso_\alpha$ in \cite{HjorthKechrisLouveau}, the $=^\alpha$ in \cite{GaoIDST}, and the $T_\alpha$ in \cite{KoerwienBRDepth}.

\begin{definition}\label{spectrum}
	$T_0$ is the theory of $(\b Z,S)$, where $S$ is the graph of the successor function.
	
	For each countable $\alpha$, $T_{\alpha+1}=J(T_\alpha)$. If $\alpha$ is a limit ordinal, $T_\alpha=\prod_{\beta<\alpha} T_\beta$.
\end{definition}

We quickly summarize the properties which are relevant to us. First, they all have Borel isomorphism relation; that is, $\iso_{T_\alpha}$ is Borel for every $\alpha$. (This is because having Borel isomorphism relation is preserved under jumps and products, as can be easily checked.)  Second, they are cofinal among such theories: if $\iso_\Phi$ is Borel, then $\Phi\borelleq T_\alpha$ for some $\alpha$, see \cite{GaoIDST} Corollary 12.2.8.  Therefore, characterizing the relationship between some $\Phi$ and the $T_\alpha$ is a reasonable way to gauge the complexity of $(\mbox{Mod}(\Phi), \iso)$. Finally:

\begin{cor}  \label{countT}
	Each $T_\alpha$ is a complete first-order theory which is short and grounded. For finite $\alpha$, $\|T_\alpha\|=\beth_\alpha$; for infinite $\alpha$, $\|T_\alpha\|=\beth_{\alpha+1}$.
	In particular, $T_\alpha\borellt T_\beta$ for all $\alpha<\beta<\omega_1$.
\end{cor}
\begin{proof}
	We first check $\alpha=0$.  Clearly $T_0$ has exactly $\aleph_0$ models up to back-and-forth equivalence.  If $M\models T_0$ in some $\b V[G]$, then either $M$ has finite dimension $n$, so is isomorphic to a model in $\b V$, or $M$ has infinite dimension, so is back-and-forth equivalent to the countable model $(\omega\times\b Z,S)$ in $\b V$.  Either way, this simultaneously shows that $T_0$ is grounded and $||T_0||=\aleph_0$, completing the proof.
	
	The rest of the proof goes immediately by induction, using Propositions~\ref{BasicJumpProposition} and \ref{BasicProductProposition}.  Groundedness follows from  part (2). Size counting (and therefore shortness) follows from part (3), which is immediate at successor stages but the limits require an argument. So let $\alpha$ be a limit ordinal.  Then $\|T_\alpha\|=\prod_\beta \|T_\beta\|$ (regardless of whether there are unsorted elements).  By the inductive hypothesis and the rules of cardinal multiplication, this is equal to $\prod_\beta \beth_\beta=\prod_{\beta} 2^{\beth_\beta}$, which in turn is equal to $2^{(\sum_{\beta}\beth_\beta)}=2^{\beth_\alpha}=\beth_{\alpha+1}$, completing the proof.
	
	The strictness of the ascending chain follows from induction and part (5).

\end{proof}

\section{Compact group actions}\label{CompactGroupSection}

In this brief section we use the technology of canonical Scott sentences and representability to analyze the effect of a continuous action of a compact group on a Polish space $X$.  In particular, we show that the quotient of $\P_{\aleph_1}(X)$ by the diagonal action of $G$ is representable.
We also show that if the group is abelian, we can bound the potential cardinality of the representation.
 In Section~\ref{OmegaStableSection} we use these results to analyze the models of the theory $\K$ and to contrast $\K$ with $\TK$.


Suppose we have a Polish group $G$ acting on a Polish space $X$. To apply our machinery to this situation we need to say what it means for the objects involved to be strongly definable:

\begin{definition}\label{sdPolishSpace}
\begin{itemize} \item A strongly definable Polish space is a sequence $(X,d, D, i)$ of strongly definable families, where persistently: $d$ is a complete metric on $X$, $D \subset X$ is dense and $i: \omega \to D$ is a bijection.
\item A strongly definable Polish group is a sequence $(G, d', D', i', \times)$ where $(G, d', D', i')$ is a strongly definable Polish space and persistently, $\times$ is a compatable group operation on $G$. 

\item Suppose $(G, d', D', i',  \times)$ is a strongly definable Polish group, $(X,d, D, i)$ is a strongly definable Polish space. Then a strongly definable continuous action of $G$ on $X$ is a strongly definable family $\cdot$ such that persistently, $\cdot \subset G \times X \times X$ is a continuous action of $G$ on $X$.
\end{itemize} 
\end{definition}

Throughout this subsection, we fix a strongly definable Polish space $(X, d, D, i)$, a strongly definable, persistently compact Polish group $(G, d', D',i', \times)$, and a strongly definable continuous action $\cdot$ of $G$ on $X$.

We also fix strongly definable families
$$\m B_n=\{U_i^n:i\in\omega\}$$
such that persistently, each $\m B_n$ is a basis for the topology on $X^n$. (For instance, take $\m B_1$ to be the balls with rational radius and center in $D$, using the enumeration of $D$ given by $i$.)


The action of $G$ on $X$ naturally gives diagonal actions on both $X^n$ and $\P(X)$ defined by $g\cdot\abar=\<g\cdot a:a\in\abar\>$ and $g\cdot A=\{g\cdot a:a\in A\}$,
respectively. Clearly, the diagonal action of $G$ takes countable subsets of $X$ to countable subsets.  For all of these spaces, let $\sim^G$ be the equivalence relation induced by $G$.

In order to understand the quotient $(\P_{\aleph_1}(X),\sim^G)$, we begin with one easy lemma that uses the fact that $G$ is compact.  This lemma is the motivation for the language we define below. 


\begin{lemma}  \label{ctblecompact}  If $A,B\in \P_{\aleph_1}(X)$, then $A\sim^G B$ if and only if there is a bijection $\sigma:A\rightarrow B$ satisfying 
$\abar\sim^G h(\abar)$ for all $\abar\in A^{<\omega}$.
\end{lemma}

	\begin{proof}
		If $g\cdot A=B$, then $\sigma:=g\mr{A}$ is as desired. For the converse, fix such a $\sigma$; we will show there is $g\in G$ inducing $\sigma$.  Let $\{a_n:n\in\omega\}$ be an enumeration of $A$, and for each $n$, let $\overline{a}_n$ be the tuple $a_0 \ldots a_{n-1}$ and let $C_n \subseteq G$ be the set of all $g \in G$ with $g \cdot \overline{a}_n = \sigma(\overline{a}_n)$. $C_n$ is closed since the action is continuous and $C_n$ is nonempty by hypothesis. Since $G$ is compact, $C = \bigcap_n C_n$ is nonempty, and clearly any $g \in C$ has $g \cdot A = B$.
	\end{proof}

We define a language $L$ and a class of $L$-structures that encode this information.
Put  $L:=\{R^n_i:i\in\omega,n\ge 1\}$, where each $R^n_i$ is an $n$-ary relation.  
Let $M_X$ be the $L$-structure with universe all of $X$, with each $R^n_i$ interpreted by
$$M_X\models R^n_i(\abar)\quad \hbox{if and only if}\quad G\cdot \abar\cap U^n_i=\emptyset$$
As notation, let $\qf_n(\abar)$ denote the quantifier-free type of $\abar\in X^n$. 
It is easily seen that to specify $\qf_n(\abar)$ it is enough to specify the set of $i \in \omega$ such that $M_X \models R^n_i(\overline{a})$. Also, $$\qf_n(\abar)=\qf_n(\bbar)\quad\hbox{if and only if} \quad G\cdot \abar=G\cdot \bbar$$
As well, note that every $g\in G$ induces an $L$-automorphism of $M_X$ given by $a\mapsto g\cdot a$.
These two observations imply that $M_X$ has a certain homogeneity -- For $\abar,\bbar\in X^n$,
$\qf_n(\abar)=\qf_n(\bbar)$ if and only if there is an automorphism of $M_X$ taking $\abar$ to $\bbar$.

For $\abar,\bbar\in X^n$ the relation $\abar\sim^G\bbar$ is absolute between $\V$ and any forcing extension $\V[H]$. To see this, note that it suffices to check that $\qf_n$ is absolute; and in turn it suffices to check that each $R^n_i$ is absolute. But $\overline{a} \in R^n_i$ iff for some or any sequence $(\overline{d}_m: m \in \omega)$ from $D^n$ converging to $\overline{a}$, we have that for large enough $m$, $D'\cdot  \overline{d}_m \cap U^n_i = \emptyset$.

It is not hard to check that the range of $\qf_n$ is analytic ($\Gamma(\overline{x})$ is in the range of $\qf_n$ iff there is a convergent sequence $(\overline{d}_m: m < \omega)$ from $D^n$ satisfying various Borel properties). Hence by Shoenfield Absoluteness, the range of $\qf_n$ is absolute.
%
%

As notation, call an $L$-structure $N\in HC$ {\em nice} if it is  isomorphic to a substructure of $M_X$.  Let $\N$ consist of all nice $L$-structures.

\begin{lemma} \label{nice}  An $L$-structure $N\in HC$ is nice if and only if for every $n\ge 1$, every quantifier-free $n$-type realized in $N$ is realized in $M_X$.
\end{lemma}

\begin{proof}  Left to right is obvious.  For the converse, choose any $N\in HC$ for which every quantifier-free $n$ type realized in $N$ is realized in $M_X$.
We construct an $L$-embedding of $N$ into $M_X$ via a ``forth'' construction using the homogeneity of $M_X$.
Enumerate the universe of $N=\{a_n:n\in\omega\}$ and let $\abar_n$ denote $\<a_i:i<n\>$.  Assuming $f_n:\abar_n\rightarrow M_X$ has been defined,
choose any $\bbar\in X^{n+1}$ such that $\qf_{n+1}(\abar_{n+1})=\qf_{n+1}(\bbar)$.  Write $\bbar$ as $\bbar_nb^*$.  As $\qf_n(\bbar_n)=\qf_n(f_n(\abar_n))$,
there is an automorphism $\sigma$ of $M_X$ with $\sigma(\bbar_n)=f_n(\abar_n)$.  Then define $f_{n+1}$ to extend $f_n$ and satisfy $f_{n+1}(a_n)=\sigma(b^*)$.
\end{proof}

Define a map $f:\P_{\aleph_1}(X)\rightarrow \N$
by $A\mapsto M_A$, the substructure of $M_X$ with universe $A$.

Our first goal is the following Theorem.

\begin{thm} \label{compact} Suppose $(X, d, D, i)$ is a strongly definable Polish space, $(G, d', D', i', \times)$ is a strongly definable, persistently compact Polish group, and $\cdot$ is a strongly definable continuous action of $G$ on $X$. Then:
\begin{enumerate}
\item  Both $\N$ and $f$ are strongly definable;
\item  Persistently, for all $A,B\in \P_{\aleph_1}(X)$, $A\sim^G B$ if and only if $f(A)\iso f(B)$;
\item  The canonical Scott sentence map $\css:(\N,\iso)\rightarrow CSS(\N)$ is a representation, where $\CSS(\N) = \{\css(N): N \in \N\}$;
\item  The quotient $(\P_{\aleph_1}(X),\sim^G)$ is representable via the composition map $\css\circ f$ that takes $A\mapsto \css(M_A)$.
\end{enumerate}
\end{thm}

\begin{proof}  (1) It is obvious that $f$ is strongly definable.

That $\N$ is strongly definable follows from Lemma~\ref{nice} and the absoluteness results mentioned above.  In particular, for any $L$-structure $N\in\HC$
that is not nice, there is some $n$ and $\abar\in N^n$ such that $\Gamma:=\qf_n(\abar)$ is not realized in $M_X$.  But then, in any forcing extension $\V[H]$, $(M_X)^{\V[H]}$
does not realize $\Gamma$, so $N$ is not nice in $\V[H]$.  As this argument relativizes to any forcing extension, $\N$ is strongly definable. 
 
For (2), if $A\sim^G B$, then  any $g\in G$ that  satisfies $g\cdot A=B$ induces a bijection between $A$ and $B$ such that $\abar\sim^G g\cdot \abar$ for all $\abar\in A^{<\omega}$.
As this implies $G \cdot \abar=G \cdot (g\cdot\abar)$, $\qf_n(\abar)$ in $M_A$ is equal to $\qf_n(g\cdot \abar)$ in $M_B$.   Thus, the action by $g$ induces an isomorphism of the $L$-structures $M_A$ and $M_B$.  Conversely, suppose $\sigma:M_A\rightarrow  M_B$
is an $L$-isomorphism.  Then  $\sigma(\abar)\sim^G\abar$ for every $\abar\in A^{<\omega}$, so $A\sim^G B$ by Lemma~\ref{ctblecompact}.

(3)  As $\N$ is strongly definable, this follows immediately from Lemmas~\ref{Karp0} and \ref{restriction}.

(4) follows immediately from (1), (2), and (3).
\end{proof}

\bigskip

As a consequence of Theorem~\ref{compact}, $||(\P_{\aleph_1}(X),\sim^G)||$ is defined.  For an arbitrary compact group action, this quotient need not be short.
Indeed, Theorem~\ref{graphs}
 gives an example where it is not.  However, if we additionally assume that $G$ is abelian, then we will see below that $||(\P_{\aleph_1}(X),\sim^G)||\le\beth_2$.
The reason for this stark discrepancy is due to the comparative simplicity of abelian group actions.  In particular, if an abelian group $G$ acts transitively on a set $S$,
then $S$ is essentially an affine copy of $G/\textrm{Stab}(a)$, where $\textrm{Stab}(a)$ is the subgroup of $G$ stabilizing any particular $a\in S$.  The following Lemma is really a restatement
of this observation.  

	\begin{lemma}\label{AbelianEqualityLemma}
	
	Suppose that an abelian group $G$ acts on a set $X$.  Then for every $n\ge 1$, if three $n$-tuples $\abar,\bbar,\cbar\in X^n$ satisfy $\abar\sim^G\bbar\sim^G\cbar$
	and
	$\overline{ab}\sim^G \overline{ac}$, then $\bbar=\cbar$.
			\end{lemma}
	\begin{proof}
		Let $g\in G$ be such that $g\o a=\o b$.  Choose $h\in G$ such that $h(\overline{ab})=\overline{ac}$.  Then in particular, $h\o a=\o a$ and $h\o b=\o c$.  From this, $gh\o a=\o b$ and $hg\o a=\o c$.  But $gh=hg$, so $\o b=\o c$, as desired.
	\end{proof}

We will show that $\|(\m W,\iso)\|\leq\beth_2$ by showing that each Scott sentence in the representation is from $L_{\beth_1^+\omega}$, and then using the fact that there are at most $\beth_2$ such sentences. In the case that the Scott sentence is satisfiable, this means the model has size at most $\beth_1$, so this is perhaps unsurprising.  We will accomplish this complexity bound by a type-counting argument; here is the notion of type we will use.

If $\phi$ is a canonical Scott sentence -- that is, $\phi\in \CSS(L)_\ptl$ -- then let $S^\infty_n(\phi)$ be the set of all canonical Scott sentences in the language $L'=L\cup \{c_0,\ldots,c_{n-1}\}$ which imply $\phi$.  We will refer to elements of $S^\infty_n(\phi)$ as types -- infinitary formulas with free variables $x_0,\ldots,x_{n-1}$, resulting from replacing each $c_i$ with a new variable $x_i$ not otherwise appearing in the formula.  It is equivalent to define $S^\infty_n(\phi)$ by forcing -- if $\b V[H]$ makes $\phi$ hereditarily countable and $M\in \b V[H]$ is the unique countable model of $\phi$, then $S^\infty_n(\phi)$ is the set $\{\css(M,\o a):\o a\in M^n\}$.  Evidently this set depends only on the isomorphism class of $M$, so by the usual argument with Lemma~\ref{product3}, this set is in $\b V$.

\begin{prop}\label{SmallBranchingSmallFormulaTheorem}
	Suppose $\phi$ is a canonical Scott sentence in a language of size at most $\kappa$, and for all $n$, $|S^\infty_n(\phi)|\leq \kappa$, where $\kappa$ is an infinite cardinal.  Then $\phi$ is a sentence of $L_{\kappa^+\omega}$.
\end{prop}
\begin{proof}
	We use the precise syntactic definition of Scott formulas from Definition~\ref{cssDefinition}.  For a moment, pass to a forcing extension $\b V[G]$ in which $\phi$ is hereditarily countable, and let $M$ be its unique countable model.  For each ordinal $\alpha$, let $S^\alpha_n(\phi)$ be the set $\{\phi^{\o a}_\alpha(\o x):\o a\in M^n\}$.  Since $M$ is (persistently above $\b V[G]$) unique up to isomorphism, and since this set is unchanged by passing to an isomorphic image of $M$, $S^\alpha_n(\phi)$ is in $\b V$ and depends only on $\phi$.  Moreover, there is a natural surjection $\pi^\alpha_n:S^\infty_n(\phi)\to S^\alpha_n(\phi)$ taking $\css(M,\o a)$ to $\phi^{\o a}_\alpha(\o x)$; each $\pi^\alpha_n$ is in $\b V$.
	
	Let $\alpha^*$ be the Scott rank of $M$.  Again, this is invariant under isomorphism, so depends only on $\phi$.  For any two sentences $\psi,\tau\in S^\infty_n(\phi)$, let $d(\psi,\tau)$ be the least $\alpha<\alpha^*$ where $\pi^{\alpha+1}_n(\psi)\not=\pi^{\alpha+1}_n(\tau)$; if there is no such $\alpha$, then $\psi=\tau$, so say $d(\psi,\tau)=\alpha^*$.  It is immediate from the construction of Scott rank that if $\alpha\le \alpha^*$, there are Scott sentences $\psi$ and $\tau$ of some arity where $d(\psi,\tau)=\alpha$; hence $d:\bigcup_n (S^\infty_n(\phi))^2\to \alpha^*+1$ is surjective.  Further, $d$ depends only on $\phi$, so by Lemma~\ref{product3}, $d\in\b V$.
	
	The rest of the proof takes place in $\b V$.  Because $\pi^\alpha_n$ is surjective and $|S^\infty_n(\phi)|\leq \kappa$, $|S^\alpha_n(\phi)|\leq\kappa$ for all $\alpha$.  Similarly, $|\bigcup_n (S_n^\infty(\phi))^2|\leq\kappa$ and $d$ is surjective, so $|\alpha^*+1|\leq\kappa$.  By induction we show that for all $\alpha\leq \alpha^*+1$, $S^\alpha_n(\phi)\subseteq L_{\kappa^+\omega}$.  The base case is trivial, since there are only $\kappa$ atomic formulas.  
	The step follows from the fact that $|S^\alpha_n(\phi)|\leq\kappa$, and the limit follows from the fact that $\alpha^*<\kappa^+$, so in both cases we need only take conjunctions and disjunctions of $\kappa$ formulas at a time.
	
	Observe that $\phi$ is precisely the following: \[ \pi_0^{\alpha^*}(\phi) \land  \bigwedge\left\{\forall \o x \left( \pi^{\alpha^*}_n(\phi^*)(\o x)\to \pi^{\alpha^*+1}_n(\phi^*)(\o x) \right)
	: n\in\omega,\phi^*\in S^\infty_n(\phi)\right\} \]
	
	Since $S^\alpha_n(\phi)\subseteq L_{\kappa^+\omega}$ for all $\alpha$ and $n$, and since they all have size at most $\kappa$, $\phi$ is in $L_{\kappa^+\omega}$, as desired.
\end{proof}

The following holds by a straightforward induction on the complexity of formulas:

\begin{lemma}\label{FewFormulasKappaPlusLemma}
	For all infinite cardinals $\kappa$ and languages $L$ of size at most $\kappa$, there are exactly $2^\kappa$ different $L_{\kappa^+\omega}$ formulas.
\end{lemma}

Now we can prove our theorem. Recall that $\sim_G$ is the diagonal equivalence relation on $\mathcal{P}_{\aleph_1}(X)$, induced by the diagonal action of $G$.
\begin{theorem}\label{compact2}
 Let $X$ be a strongly definable Polish space, let $G$ be a strongly definable, persistently compact abelian group, and suppose $\cdot$ is a strongly definable continuous action of $G$ on $X$.  Suppose all this holds persistently.
 Then $\|(\m P_{\aleph_1}(X),\sim_G)\|\leq \beth_2$.
\end{theorem}


%
\begin{proof}
	We use Proposition~\ref{SmallBranchingSmallFormulaTheorem} to show that $\CSS(\m W)_\ptl\subseteq L_{\beth_1^+\omega}$; then by Lemma~\ref{FewFormulasKappaPlusLemma}, we have that $|\CSS(\m W)_\ptl|\leq\beth_2$, as desired.  So let $\phi\in \CSS(\m W)_\ptl$ be arbitrary; it is enough to show that $|S^\infty_n(\phi)|\leq\beth_1$.
	
	For each $n$, let $\qf_n(\phi)$ be the set of quantifier-free $n$-types which are consistent with $\phi$.  We have a surjective map $\pi_n:S^\infty_n(\phi)\to \qf_n(\phi)$ sending $\psi(\o x)$ to the set of quantifier-free formulas in $\o x$ which it implies.  For any $p\in \qf_n(\phi)$, let $S^\infty_n(\phi,p)$ be $\pi^{-1}(p)$, the set of $\psi\in S^\infty_n(\phi)$ where $\pi_n(\psi)=p$.  (All of these definitions have taken place in $\mathbb{V}$.) Since the language is countable, $|\qf_n(\phi)|\leq\beth_1$.  Thus it is sufficient to show that for all $p$, $|S^\infty_n(\phi,p)|\leq \beth_1$. 
	
	Now we take advantage of the fact that $G$ is abelian:
	
	\begin{claim}
		Suppose $p^*\in \qf_{2n}(\phi)$ is such that $p^*\mr{[0,n)}=p^*\mr{[n,2n)}=p$.  Suppose that $\psi^*,\tau^*\in S^\infty_{2n}(\phi)$ both complete $p^*$. Further, suppose $\psi^*\mr{[0,n)}=\tau^*\mr{[0,n)}$.  Then $\psi^*=\tau^*$.
	\end{claim}
	\begin{claimproof}
		Pass to a forcing extension $\b V[H]$ in which $\phi$ is hereditarily countable, and let $M$ be its unique countable model.  By Theorem~\ref{compact}, we may assume $M=M_A$ for some $A\in \m P_{\aleph_1}(X)^{\b V[H]}$.  Choose some tuples $(\overline{a_0 a_1})$ and $(\overline{b_0 b_1})$ from $M^{2n}$ where $\css(M,\overline{a_0 a_1})=\psi^*$ and $\css(M,\overline{b_0 b_1})=\tau^*$.  By assumption $\css(M,\o a_0)=\css(M,\o b_0)$, so we may assume $\o a_0=\o b_0$.  Since all of the tuples $\o b_0$, $\o a_1$, and $\o b_1$ have the same quantifier-free type, they are in the same $G$-orbit, and similarly with $\o b_0\o a_1$ and $\o b_0\o b_1$.  Thus Lemma~\ref{AbelianEqualityLemma} applies directly to the triple $(\o b_0,\o b_1,\o a_1)$, so in particular $\o b_1=\o a_1$.  Thus $\psi^*=\css(M,\overline{b_0 a_1})=\css(M,\overline{b_0 b_1})=\tau^*$, as desired.
	\end{claimproof}
	
	\medskip
	
	Fix some $\psi\in S^\infty_n(\phi,p)$, and define $\Gamma(\psi)$ to be the set of all $p^*\in \qf_{2n}(\phi)$ such that $p^*\mr{[0,n)}=p^*\mr{[n,2n)}=p$ and such that for some $\psi^*\in S^\infty_{2n}(\phi,p^*)$, $\psi^*\mr{[0,n)}=\psi$.
	
	By the Claim, if $p^*\in \Gamma(\psi)$, there is a unique $\psi^*\in S^\infty_{2n}(\phi)$ where $\pi_{2n}(\psi^*)=p^*$ and $\psi^*\mr{[0,n)}=\psi$.  So define $F(p^*)$ to be $\psi^*\mr{[n,2n)}$.  Evidently $|\Gamma(\psi)|\leq\beth_1$, so it is enough to show that $F:\Gamma(\psi)\to S^\infty_n(\phi, p)$ is surjective.
	
	But this is almost immediate.  Fix any $\tau\in S^\infty_n(\phi, p)$ and let $\b V[H]$ be a forcing extension in which $\phi$ is hereditarily countable, and let $M$ be its unique countable model;  as before, we may assume $M=M_A$ for some $A\in \m P_{\aleph_1}(X)^{\b V[H]}$.  Choose any $\o a\in A^n$ where $\css(M,\o a)=\psi$ and any $\o b\in A^n$ where $\css(M,\o a)=\tau$.  Finally, let $p^*$ be the quantifier-free type of $\overline{ab}$ in $M$.  Clearly $p^*\in \Gamma(\psi)$ and $F(p^*)=\tau$.
\end{proof}	

This theorem will be crucial in Section~\ref{OmegaStableSection}.

\section{Refining Equivalence Relations}\label{REFsection}

We begin by defining an incomplete first-order theory $\REF$.  Its language is $L=\{E_n:n\in\omega\}$ and its axioms posit:
\begin{itemize}
\item  Each $E_n$ is an equivalence relation;
\item  $E_0$ has a single equivalence class; that is, we consider $xE_0y$ to be universally true;
\item  For all $n$, $E_{n+1}$ refines $E_n$; that is, every $E_n$-class is a union of $E_{n+1}$-classes.
\end{itemize}
The theory $\REF$ is very weak, which makes the generality of the following proposition surprising.

\begin{prop}  \label{gr} $\REF$ is grounded.
\end{prop}

\begin{proof}  
We begin with an analysis of an arbitrary model $M$ of $\REF$.  
As notation, for any $a\in M$ and $n\in\omega$, let $[a]_n$ denote
the equivalence class of $a$, i.e., $\{b\in M:M\models E_n(a,b)\}$.  
As the equivalence relations refine each other, the classes $T(M)=\{[a]_n:a\in M, n\in\omega\}$ form an $\omega$-tree, ordered by $[a]_n\le [b]_m$
if and only if $n\le m$ and $[b]_m\subseteq [a_n]$.
Next,
let $E_\infty$ be the equivalence relation given by $E_\infty(a,b)$ if and only if $E_n(a,b)$ for every $n\in\omega$.  Let $[a]_\infty$ be the $E_\infty$-class of $a$.
Then $M/E_\infty$ can be construed as a subset of  the branches $[T(M)]$ of $T(M)$.  
As we are interested in determining models up to back-and-forth equivalence (as opposed to isomorphism), the following definition is natural.

For each $a\in M$, let the {\em color of $a$}, $c(a)\in(\omega+1)\setminus\{0\}$ be given by
$$c(a)=
 \left\{\begin{array}{ll}
|[a]_\infty| &\mbox{if $[a]_\infty$ is finite}\\
\omega & \mbox{if $[a]_\infty$ is infinite}
\end{array}
\right. $$

Next, we  describe some expansions of $M$ to larger languages.
For each $n\in\omega$, let $L_n=L\cup\{U_i:i\le n\}$, where the $U_i$'s are distinct unary predicates.  Given any $M\models \REF$, $n\in\omega$, and $a\in M$,
let $M_n(a)$ denote the $L_n$-structure $(M,[a]_0,\dots,[a]_{n})$, i.e., where each predicate $U_i$ is interpreted as $[a]_i$.

We now exhibit some invariants, which we term the {\em data of $M$}, written $D(M)$ which we will see only depend on the $\equiv_{\infty,\omega}$-equivalence class of $M$.

For each $n\in\omega$, let
$$I_n(M)=\{\css(M_n(a)):a\in M\}.$$
We combine the sets $I_n(M)$ into a tree $(I(M),\le)$ where
$I(M)=\bigcup_{n\in\omega} I_n(M)$ and, for $\sigma_n\in I_n(M)$ and $\psi_m\in I_m(M)$, we say $\sigma_n\le \psi_m$ if and only if
$n\le m$ and $\psi_m\vdash \sigma_n$.  That is, if in any forcing extension the reduct of any model of $\psi_m$ to $L_n$ is a model of $\sigma_n$.  Then clearly $(I(M),\le)$ is an $\omega$-tree.

Continuing, for each $n>0$ and $\sigma_n\in I_n(M)$, let  the multiplicity of $\sigma_n$, $\mult_M(\sigma_n)\in (\omega+1)\setminus\{0\}$, be given by:
$\mult_M(\sigma_n)=k<\omega$ if $k$ is maximal such that there are elements $\{b_i:i<k\}\subseteq M$ such that
$$ \bigwedge_{i<j<k} \bigg[ E_{n-1}(b_i,b_j)\wedge \neg E_n(b_i,b_j)\bigg]\wedge\bigwedge_{i<k} \css(M_n(b_i))=\sigma_n$$
and let $\mult_M(\sigma_n)=\omega$ if there is an infinite family $\{b_i:i<\omega\}$ as above.

Now, each $a\in M$ induces a canonical sequence  $Seq_M(a):=\<\css(M_n(a)):n\in\omega\>$, which is clearly a branch through the tree $I(M)$, and depends only on $\css(M, a)$.
  Let $Seq(M)=\{Seq_M(a):a\in M\}$. So $Seq(M)\subseteq [I(M)]$, the set of branches of $I(M)$.  Finally, for any $s\in Seq(M)$, we define the color spectrum of $s$ as $Sp_M(s):=\{c(a):Seq_M(a)=s\}$.  Thus, each $Sp_M(s)$ is a non-empty subset of $(\omega+1)\setminus \{0\}$.
  
  \medskip
  
  Define the {\em data of $M$}, $D(M):=\<(I(M),\le), \mult_M,Seq(M),Sp_M\>$.

  \medskip

  \begin{claim1}  For any $M,N\models \REF$, $M\equiv_{\infty,\omega} N$ if and only if $D(M)=D(N)$.
  \end{claim1}
  
  \begin{claimproof}  First, note that if $D(M)=D(N)$, then as the trees $(I(M),\le)$ and $(I(N),\le)$ are equal, they have the same root, so $\css(M_0(a))=\css(N_0(b))$
  for some/every $a\in M, b\in N$.  So $M\equiv_{\infty,\omega} N$.
  
  For the forward direction, it is easy to check that $D(M)$ only depends on the isomorphism type of $M$, and also that $D$ is absolute to forcing extensions. Hence if $M \equiv_{\infty \omega} N$, then pass to a forcing extension $\mathbb{V}[G]$ in which $M \cong N$; then we get $(D(M))^{\mathbb{V}} = (D(M))^{\mathbb{V}[G]} = (D(N))^{\mathbb{V}[G]} = (D(N))^{\mathbb{V}}$.
  \end{claimproof}
  
  \bigskip
 
 To begin the proof of groundedness,  choose any $\sigma\in \CSS(\REF)_\ext$.  Choose any forcing extension $\V[G]$ of $\V$ in which $\sigma\in HC^{\V[G]}$ and hence
 $\sigma\in \CSS(\REF)^{\V[G]}$.  Choose any model $M\in \V[G]$ with $\css(M)=\sigma$.  Working in $\V[G]$, compute $D(M)$, the data of $M$.  
 However, in light of Claim~1, $D(M)$ only depends on $\sigma$, and so by Lemma~\ref{product3} $D(M) \in \mathbb{V}$.  As $\sigma$ is fixed, for the remainder of the argument we write
 $$D=\<(I,\le),\mult,Seq,Sp\>.$$

\medskip

To complete the proof of the Proposition, we work in $\V$ and  `unpack' the data $D$ to construct an $L$-structure $N\in\V$ such that in $\V[G]$, $M\equiv_{\infty,\omega} N$.
 Once we have this,  as $\sigma=\css(M)$, it follows that $N\models \sigma$ and so $N$ witnesses that $\sigma\in \CSS(\REF)_\sat$. 
 That is, the proof of groundedness will be finished once we establish the following Claim.
 
 \begin{claim2}  There is an $L$-structure $N \in \mathbb{V}$ such that $\V[G]\models \hbox{`$N\equiv_{\infty,\omega} M$'}$.
 \end{claim2}
 
 \begin{claimproof} Before beginning the `unpacking' of $D$, we note some connections between $M$ and $D$ that are not part of the data.
 First, there is a surjective tree homomorphism $h:T(M)\cup M/E_{\infty} \rightarrow I\cup Seq$ given by $h([a]_n)=\css(M_n(a))$ for $n\in\omega$ and
 $h([a]_\infty)=\<\css(M_n(a)):n\in\omega\>$. Note that for each $s \in Seq$ and each $k \in Sp(s)$, $\{[a]_{\infty}: h(a) = s \mbox{ and } c(a) = k\}$ is dense in $h^{-1}(s)$. The following relationship between $M$ and $h$ follows quickly:
 
 \begin{quotation}  $(\star)_{M,h}$:  For every $n\ge 1$, $s\in Seq$, $k\in Sp(s)$, and $a\in M$ such that $h([a]_{n-1})=s(n-1)$,
 there are pairwise $E_n$-inequivalent $\{d_i:i<\mult(s(n))\}\subseteq M$ such that
 $$\bigwedge_{i<\mult(s(n))} \bigg[ E_{n-1}(d_i, a) \wedge h([d_i]_\infty)=s\wedge c(d_i)=k\bigg]$$
 \end{quotation}
 
 We also identify two species of elements of $Seq$.  Call $s\in Seq$ of {\em isolated type} if there is $n\in\omega$ such that $\mult(s(m))=1$ for every $m\ge n$
 and of {\em perfect type} otherwise.  The latter name is apt, as $h^{-1}(s)$ is perfect (has no isolated points) whenever $s$ is not of isolated type.
 We argue that if $s\in Seq$ is of isolated type, then $Sp(s)$ is a singleton.  Indeed, choose $n$ such that $\mult(s(m))=1$ for every $m\ge n$ and choose $a,b\in M$
 such that $h([a]_\infty)=h([b]_\infty)=s$.  We will show that $c(a)=c(b)$.  To see this, by applying $(\star)_{M,h}$ at level $n+1$ with $k=c(b)$, get $d\in M$ such that
 $E_n(a,d)$, $h([d]_\infty)=s$, and $c(d)=c(b)$.  But now, as $h([a]_\infty)=h([d]_\infty)=s$, the choice of $n$ implies that $E_\infty(a,d)$.  Thus,
 $c(a)=c(d)=c(b)$ as required.


  We begin `unpacking' $D$ by inductively constructing an $\omega$-tree $(J,\le)$ and a surjective tree homomorphism $h':(J,\le)\rightarrow (I,\le)$.  
 Begin the construction of $J=\bigcup_{n\in\omega} J_n$ by taking  $J_0=\{\rho_0\}$ to be a singleton  and defining $h'(\rho_0)=\sigma$.
 Suppose the $n$th level $J_n$ has been defined, together with $h':\bigcup_{j\le n} J_j\rightarrow \bigcup_{j\le n} I_j$.   For each $\rho_n\in J_n$,
 we define its immediate successors 
 $Succ_J(\rho_n)$ as follows.  Look at $Succ_I(h'(\rho_n))\subseteq I_{n+1}$.  For each $\sigma_{n+1}\in Succ_I(h'(\rho_n))$,
 choose a set $A_{n+1}(\sigma_{n+1})$ of  cardinality $\mult(\sigma_{n+1})\in(\omega+1)\setminus\{0\}$ such that the sets $A_{n+1}(\sigma_{n+1})$  are pairwise disjoint.
 Let $$Succ_J(\rho_n):=\bigcup \{A_{n+1}(\sigma_{n+1}):\sigma_{n+1}\in Succ_I(h(\rho_n))\}$$
 and put $J_{n+1}:=\bigcup\{Succ_J(\rho_n):\rho_n\in J_n\}$.  
We extend $h'$ by $h'(\rho)=\sigma_{n+1}$ for every $\rho\in A_{n+1}(\sigma_{n+1})$.


Now, having completed the construction of $(J,\le)$ and the tree homomorphism $h':(J,\le)\rightarrow (I,\le)$, there is a unique extension  (which we also call $h'$)
$h':[J]\rightarrow[I]$ from the branches of $J$ to the branches of $I$ such that $h'(\eta)=s$ if and only if $h'(\eta\mr{n})=s(n-1)$ for every $n\in\omega$.

The universe of the $L$-structure $N$ we are building will be a subset of $(h')^{-1}(Seq)\times(\omega+1)$ and for $(\eta,i),(\nu,j)\in N$,
we will interpret $E_n$ by
$$E_n((\eta,i),(\nu,j))\quad\hbox{if and only if} \quad \eta\mr{n}=\nu\mr{n}.$$
In particular, we will have $[(\eta,i)]_\infty=\{(\eta,j):(\eta,j)\in N\}$.  To finish our description of $N$, we must assign a `color' to each element of $(h')^{-1}(Seq)$.
Fix $s$; we assign colors to $(h')^{-1}(s)$.  First, if $s$ is of isolated type, then from above, we know that $Sp(s)=\{k\}$ for a single
color
$k\le\omega$.  Accordingly, put elements $\{(\eta,i):i<k\}$ into the universe of $N$ for every $\eta$ satisfying $h'(\eta)=s$.
For each $s\in Seq$ that is not of isolated type, note that $(h')^{-1}(s)$ has no isolated points.  Thus, we can choose a partition $(h')^{-1}(s)=\bigcup D_k(s)$
into disjoint dense subsets indexed by colors $k\in Sp(s)$.  Then, for each $\eta\in D_k(s)$ put elements $\{(\eta,i):i<k\}$ into the universe of $N$.
This completes our construction of the $L$-structure $N\in\V$, and 
it is easily verified that this construction entails $(\star)_{N,h'}$.

We now work in $\V[G]$ and demonstrate that $M\equiv_{\infty,\omega} N$.  Indeed, all that we need for this is that in $\V[G]$, both $(\star)_{M,h}$ and $(\star)_{N,h'}$ hold.
Let $\F$ consist of all $(\abar,\bbar)$ such that $\lg(\abar)=\lg(\bbar)$, $\abar$ from $M$, and $\bbar$ from $N$ that satisfy for each $i<\lg(\abar)$,
$c(a_i)=c(b_i)$ and $h([a_i]_\infty)=h'([b_i]_\infty)$; and for each $n\in \omega$, $i<j<\lg(\abar)$, $M\models E_n(a_i,a_j)$ if and only if $N\models E_n(b_i,b_j)$ and $a_i = a_j$ if and only if $b_i = b_j$.

To see that $\F$ is a back-and-forth system, choose any $(\abar,\bbar)\in\F$ and choose any $a^*\in M$.  We will find $b^*\in N$ such that $(\abar a^*,\bbar b^*)\in\F$,
and the argument in the other direction is symmetric.  If $\lg(\overline{a}) = 0$, or if $a^*\in\abar$, it is obvious what to do, so assume this is not the case.  If $E_{\infty}(a^*,a_i)$ for some $i$,
then as $c(a_i)=c(b_i)$, we can find $b^*\not\in\bbar$ such that $E_{\infty}(b^*,b_i)$ which suffices.

Now assume that $\neg E_{\infty}(a^*,a_i)$ holds for each $i$.
Let $k=c(a^*)$ and $s=h([a^*]_\infty)$. Let $n>0$ be least such that $\neg E_n(a^*,a_i)$ for all $i$.   Let $A_1=\{a_i: E_{n-1}(a^*,a_i)\}$ and let $B_1$ be the associated subset of $\bbar$.  By the axioms of $\REF$
it suffices to find $b^*\in N$ such that $c(b^*)=k$, $h'([b^*]_\infty)=s$, $E_{n-1}(b^*,b)$ for some/every $b\in B_1$, but $\neg E_n(b^*,b)$ for every $b\in B_1$.
To find such an element, let 
$$A_2=\{a\in A_1:\ \hbox{there is some $a'\in[a]_n$ such that $c(a')=k$ and $h([a']_\infty)=s$}\}$$
Let $A_3\subseteq A_2$ be any maximal, pairwise $E_n$-inequivalent subset of $A_2$ and let $\ell=|A_3|$.
The set $\{a^*\}\cup A_3$ witnesses that $\mult(s(n))>\ell$. [More precisely, for each $a\in A_3$, choose $a'\in[a]_n$ with $c(a')=k$ and
$h([a']_\infty)=s$.  Then $\{a^*\}\cup\{a':a\in A_3\}$ witnesses  $\mult(s(n))>\ell$.]
Let $B_3$ be the associated subset of $\bbar$; so $|B_s| = \ell$.

Choose $a_i \in A_1$. Then by $(\star)_{N,h'}$, applied at $b_i$ (noting that $[b_i]_{n-1} = s(n-1)$), choose a family $\{d_i:i<\mult(s(n))\}$ as there.  By pigeon-hole choose an $i^*<\mult(s(n))$ such that $\neg E_n(d_{i^*},b)$ holds for all $b\in B_3$.
It is easily checked that $d_{i^*}$ is a possible choice for $b^*$.   As noted above, this completes the proof of the Claim.
  \end{claimproof}
  
  In particular, $N\models\sigma$, establishing groundedness.
  \end{proof}
  

%

	We now turn our attention to two classical complete  theories extending $\REF$.  These are often given as first examples in stability theory.
	 We  denote them by $\REF(\inf)$ and $\REF(\bin)$, respectively.
	$\REF(\bin)$ is the extension of $\REF$ asserting that for every $n$, $E_{n+1}$ partitions each $E_n$-class into two $E_{n+1}$-classes, while $\REF(\inf)$ asserts
	that for all $n$, $E_{n+1}$ partitions each $E_n$-class into infinitely many $E_{n+1}$-classes.
	
The following facts are well known.

\begin{fact}  Both $\REF(\bin)$ and $\REF(\inf)$ are complete theories that admit quantifier elimination.
\begin{itemize}
\item  $\REF(\bin)$ is superstable but not $\omega$-stable; and
\item  $\REF(\inf)$ is stable, but not superstable.
\end{itemize}
\end{fact}
	
	We will see below that these theories have extremely different countable model theory.  Both are similar in that the isomorphism relation $\iso$ is not Borel.
	However, 
	it turns out that $\REF(\inf)$ is Borel complete, and indeed, is $\lambda$-Borel complete for every $\lambda$.  That is, $\REF(\inf)$ the class of models is maximally complicated, both at the countable level as well as  at every uncountable level.  
	On the other hand, $\REF(\bin)$ is far from being Borel complete.  In fact, its class of countable models embeds  $T_2$ but not $T_3$.

	\subsection{Finite Branching}
	In this subsection we show that $T_2 \leq_B \REF(\bin)$ but $T_3 \not \leq_B \REF(\bin)$, and that the isomorphism relation of $\REF(\bin)$ is not Borel. 
    
    For the following, it would be inconvenient to work with $T_2$ directly. Instead, let $F_2$ be the equivalence relation on $(2^\omega)^\omega$ defined by: $(x_n: n \in \omega) F_2 (y_n: n \in \omega)$ iff $\{x_n: n \in \omega\} = \{y_n: n \in \omega\}$. Then the quotient $(2^\omega)^\omega/F_2$
     is in natural bijection with $\mathcal{P}_{\aleph_1}(2^\omega) \backslash \{\emptyset\}$, so we think of $F_2$ as representing countable sets of reals. It is not hard to check that $(\mbox{Mod}(T_2), \cong)$ is Borel bireducible with $((2^\omega)^\omega, F_2)$. So for $T$ a theory, showing that $T_2 \leq_B T$ is the same thing as showing $F_2 \leq_B T$.

	\begin{thm}\label{REFCodesIso2Theorem}
		$T_2\borelleq \REF(\bin)$.
	\end{thm}
	\begin{proof}
		Begin by building a special countable model $M$ of $\REF(\bin)$.  Let $S$ be the set of sequences from $2^\omega$ which are eventually zero, and fix a bijection $c:S\to \b N$.  Let $M$ be the set of all $(\eta,n)$ where $\eta\in S$ and $n < c(\eta)$.  As usual, say $(\eta_1,n_1)E_m(\eta_2,n_2)$ holds if and only if $\eta_1$ and $\eta_2$ agree on the first $m$ places.  Clearly $M$ is a model of $\REF(\bin)$, and the color of $(\eta,n)$ is exactly $c(\eta)$; observe that no element has color $\aleph_0$. (Recall that the color of $a$ is the cardinality of $[a]_{\infty}$.) We will construct our models as superstructures of $M$, whose new elements all have color $\aleph_0$ and are not $E_\infty$-equivalent to any element of $M$.
		
        Let $X \subseteq (2^\omega)^\omega$ be the set of all $(x_n: n \in \omega)$ such that each $x_n \not \in S$. Then $(X, F_2 \restriction_X) \cong_B ((2^\omega)^\omega, F_2)$, via any Borel bijection between $2^\omega$ and $2^\omega \backslash S$.  (By Corollary 13.4 and Theorem 4.6 of \cite{KechrisDST}, any two uncountable Borel sets are in Borel bijection.) So it suffices to show that 
        $(X, F_2 \restriction_X) \leq_B \REF(\bin)$. 
		Given $I \subseteq 2^\omega \backslash S$ countable, let $M_I$ be the $L$-structure extending $M$ with universe $M\cup(I\times\omega)$, where again,   $(\eta_1,n_1)E_m(\eta_2,n_2)$ holds if and only if $\eta_1\mr{m}=\eta_2\mr{m}$.
        
        It is not hard to check that one can define a Borel map $f: X \to \Mod(\REF(\bin))$, such that for all $\overline{x} = (x_n: n \in \omega) \in X$, $f(\overline{x}) \cong M_{\{x_n: n \in \omega\}}$. Given that, it suffices to show that for all distinct $I, J \subseteq 2^\omega \backslash S$ countable, $M_I \not \cong M_J$.
        
		So suppose $M_I \cong M_J$, say via $g: M_I \to M_J$. I aim to show that for all $(\eta, n) \in M_I$,  $g(\eta, n) = (\eta, n')$ for some $n' < \omega$. This suffices to show $I = J$, since then $I = \{\eta: (\eta, n) \in M_I \mbox{ for all } n\} = \{\eta: (\eta, n) \in M_J \mbox{ for all } n\} = J$.
        So let $(\eta, n) \in M_I$; write $g(\eta, n) = (\tau, n')$. I show for each $m < \omega$ that $\eta \restriction_m =  \tau \restriction_m$. Indeed, pick $\nu \in S$ such that $\nu \restriction_m = \eta \restriction_m$. Then $g(\nu, 0) = (\nu, k)$ for some $k < c(\tau)$, since $g([(\nu, 0)]_{\infty})$ is the unique $E_{\infty}$-class of $M_J$ of size $c(\nu)$. Then since $((\eta, n) E_m (\tau, 0))^{M_I}$, we have $((\tau, n') E_m (\nu, k))^{M_J}$. Hence $\tau \restriction_m = \nu \restriction_m = \eta \restriction_m$.
	\end{proof}
	
	To show that $T_3$ does not embed into $\REF(\bin)$, we clarify  $\equiv_{\infty\omega}$-equivalence on a slightly wider class of $L$-structures.
	Let $\REF(\fin)$ denote the sentence of $L_{\omega_1,\omega}$ extending $\REF$ asserting that every $E_{n+1}$-class partitions every
	$E_n$-class into finitely many $E_{n+1}$-classes.  
	
	\begin{lemma} \label{subm} Every model $M$ of $\REF(\fin)$ has an $\equiv_{\infty\omega}$-equivalent  submodel $N\subseteq M$ 
		 of size at most $\beth_1$.
	\end{lemma}
	\begin{proof}
		For each $E_{\infty}$-class $[a]_\infty\subseteq M$, let 
		
		$$B([a]_\infty)=
 \left\{\begin{array}{ll}
[a]_\infty &\mbox{if $[a]_\infty$ is countable}\\
\hbox{any countably infinite subset of $[a]_\infty$}& \mbox{if $[a]_\infty$ is uncountable}
\end{array}
\right. $$ 
and let $N$ be the substructure of $M$ with universe $\bigcup\{B([a]_\infty):a\in M\}$.  It is easily seen that $N\equiv_{\infty,\omega} M$.  That $N$ has size at most continuum
follows from the finite splitting at each level.
	\end{proof}
	
	Combined with groundedness, this gives us the nonembedding result we wanted:
	
	\begin{thm}\label{REFcountingTheorem}
		$||\REF(\bin)||=||\REF(\fin)||=I_{\infty,\omega}(\REF(\bin))=I_{\infty,\omega}(\REF(\fin))=\beth_2$.  
		In particular,  both $\REF(\bin)$ and $\REF(\fin)$ are short and $T_3\not\borelleq \REF(\bin),\REF(\fin)$.
	\end{thm}
	\begin{proof}
    	Recall by Corollary~\ref{countT} that $||T_2|| = \beth_2$ and $||T_3||=\beth_3$.
		Since $T_2\borelleq\REF(\bin)$, $\beth_2=||T_2||\le||\REF(\bin)||$.  On the other hand, since $\REF$ is grounded,  $||\REF(\fin)||=I_{\infty,\omega}(\REF(\fin))$ but the latter cardinal is bounded above by $\beth_2$ by Lemma~\ref{subm}.  Thus, all four ca\textit{}rdinals are equal to $\beth_2$.  So, by definition, both $\REF(\bin)$ 
		and $\REF(\fin)$ are short.
		As $||T_3||=\beth_3$, the nonembeddability of $T_3$ into either class follows from Theorem~\ref{translate}(2).
		\end{proof}
			
	Finally, we show that isomorphism for $\REF(\bin)$ is not Borel.
	
	\begin{thm}\label{REFnotBorelTheorem}
		Isomorphism on $\REF(\bin)$ is not Borel.
	\end{thm}
	
	
	\begin{proof}

				It is commonly known -- see for example  Theorem 12.2.4  of \cite{GaoIDST} -- that isomorphism is Borel if and only if, for some $\alpha$, $\equiv_\alpha$ is sufficient to decide isomorphism. Since $\equiv_0$ is implied by $\equiv$ and $\REF(\bin)$ is a complete theory with more than one model, $\equiv_0$ does not decide isomorphism.  We proceed by induction with a combined step and limit induction step.  So suppose $\alpha_0\leq \alpha_1\leq\cdots$ are such that each $\equiv_{\alpha_n}$ does not decide isomorphism.  That is, for each $n$, there is a pair $A_n$, $B_n$ of  countable models of $\REF(\bin)$ which are nonisomorphic but where $A_n\equiv_{\alpha_n}B_n$.  Let $\alpha=\sup\{\alpha_n+1:n\in\omega\}$.  We will construct a pair  (indeed, a large family) of  countable models of $\REF(\bin)$ that are pairwise $\equiv_\alpha$-equivalent but not isomorphic.  This is sufficient.

        Recall that among countable models $M$ of \REF(\bin), the {\em color} of an element $a\in M$ is the size of its $E_\infty$-class
        $[a]_\infty$.   By adding an element to each finite $E_\infty$ class occurring in $A_n$, $B_n$, respectively, 
        we can suppose the color ``$1$'' does not occur in any of the $A_n$'s, $B_n$'s.
        
        
        Let $\mathfrak{C} \model \REF(\bin)$ be the model with universe $2^\omega \times \omega$, where as usual $(\eta, n) E_k (\tau, m)$ iff $\eta \restriction_k = \tau \restriction_k$. $\mathfrak{C}$ will serve as a `monster model' of sorts; in particular we can suppose each $A_n, B_n$ are (elementary) substructures of $\mathfrak{C}$.
        
        We begin by forming a single countable model $M\preceq\mathfrak{C}$ that encodes all of the complexity of the models $A_n,B_n$.  For $s\in 2^{<\omega}$, let
        $(A_n)^s$ be a `shift' of $A_n$ by $s$.  Formally, $(A_n)^s=\{(s\frown\eta,j):(\eta,k)\in A_n\}$ and we define $(B_n)^s$ analogously.   Whereas the substructures $A_n$ and
        $(A_n)^s$  of $\mathfrak{C}$ are certainly not elementarily equivalent, the relationships between $A_n$ and $B_n$ are maintained.  That is, if  $\lg(s)=\lg(t)$, then for any $n$,
        $(A_n)^s \equiv_{\alpha_n} (B_n)^t$, but because $A_n\not\cong B_n$, there is no elementary bijection $f:(A_n)^s\rightarrow (B_n)^t$.
        
        As notation, for $i\in\{0,1\}$ and $n\in\omega$, let $S^i_n$ be the subset of $2^{2n+2}$ satisfying 
        \begin{itemize}
        \item  $s(j)=0$ for every odd $j<2n$;
        \item  $s(2n)=i$; and
        \item  $s(2n+1)=1$.
        \end{itemize}
        Note that  not only are the sets $S^i_n$ disjoint, but in fact, $S^*:=\bigcup\{S^i_n:i\in\{0,1\},n\in\omega\}$ is an antichain on $2^{<\omega}$.
        Let $$M := \bigcup_{n < \omega, s \in S^0_n} (A_n)^s \cup \bigcup_{n < \omega, s \in S^1_n} (B_n)^s$$ 
        
        It is readily checked that $M\preceq\mathfrak{C}$.  
      Because every element of every $A_n$, $B_n$ has color distinct from 1, no element of $M$ has color 1 either.
       As notation, we refer to the subsets $(A_n)^s$ and $(B_n)^s$ as the {\em $s$-bubbles} of $M$.  Obviously, for a specific choice of $s$, $M$ contains only one of $(A_n)^s$
       or $(B_n)^s$.  We write $M(s)$ for this $s$-bubble.
       
       For each $x\in 2^\omega$, let $x^*\in 2^\omega$ be defined by $x^*(j)=0$ if $j$ is odd and $x^*(j)=x(j/2)$ if $j$ is even.
       For each countable, dense subset $X\subseteq 2^\omega$, let
       $$M_X=M\cup \{(x^*,0):x\in X\}\quad\hbox{and let}\quad S^*_X=S^*\cup\{x^*:x\in X\}$$
       Clearly, $M\preceq M_X\preceq \mathfrak{C}$ and an element $c\in M_X$ has color 1 if and only if $c\not\in M$. Write $M_X(s) = M(s)$ for $s \in S^i_n$.
       
       \begin{claim1}
			Let $X,Y\subset 2^\omega$ be countable and dense.  Then $M_X\iso M_Y$ if and only if $X=Y$.
		\end{claim1}
		\begin{claimproof}
			If $X=Y$ then $M_X=M_Y$.  On the other hand, suppose $X$ and $Y$ are dense and $f: M_X \cong M_Y$. We claim that for all $\eta \in X$, $f(\eta^*, 0) = (\eta^*, 0)$. This suffices, since then $X = \{\eta: (\eta^*, n) \in M_X \mbox{ iff } n = 0\} \subseteq \{\eta: (\eta^*, n) \in M_Y \mbox{ iff } n = 0\} = Y$ and by symmetry $Y \subseteq X$.
            
            So fix $\eta \in X$ and write $f(\eta^*, 0) = (\tau^*, 0)$ where $\tau \in Y$ ($f(\eta^*, 0)$ must be of this form since it is of color $1$ in $M_Y$). Suppose towards a contradiction that $\eta \not= \tau$; let $n$ be least such that $\eta(n) \not= \tau(n)$. Let $s = \eta^* \restriction_{2n+1}\,^\frown(1)$ and let $t = \tau^* \restriction_{2n+1}\,^\frown (1)$.  
            Then our purported isomorphism $f$ would induce an elementary permutation  between  $(A_n)^s$  and $(B_n)^t$ (or between $(A_n)^t$ and $(B_n)^s$), which is impossible
            since $A_n\not\cong B_n$.
            
		\end{claimproof}

 By contrast, we have:

		\begin{claim2}
			Let $X,Y\subset 2^\omega$ be countable and dense.  Then $M_X\equiv_\alpha M_Y$.
		\end{claim2}
		\begin{claimproof}
        We recall that $M_X \equiv_\alpha M_Y$ iff Player II has a winning strategy in the following game $\mathcal{G}(M_X, M_Y, \alpha)$:
        
        Players I and II alternate moves. On Player I's $n$-th turn, he either plays a pair $(a_n, \beta_n)$ where $a_n \in M_X$ and $\beta_n$ is an ordinal with $\alpha > \beta_0 > \ldots > \beta_n$, or else he plays a pair $(b_n, \beta_n)$, where $b_n \in M_Y$ and $\beta_n$ is an ordinal with $\alpha > \beta_0 > \ldots > \beta_n$. (Really Player I should also specify which of $M_X$, $M_Y$ he is playing in, but no harm is caused by suppressing this). On Player II's $n$-th turn, she plays either $b_n \in M_Y$ or $a_n \in M_X$, depending on Player I's move; she is required to make sure that $(a_0, \ldots, a_n) \mapsto (b_0, \ldots, b_n)$ is partial elementary from $M_X$ to $M_Y$. This specifies the game, since Player I cannot survive indefinitely.
        
        Now for each $n < \omega$, we are assuming that $A_n \equiv_{\alpha_n} B_n$, where $(\alpha_n: n < \omega)$ is increasing (possibly not strictly), such that $\alpha = \sup \{\alpha_n+1: n < \omega\}$. Fix a winning strategy $\Gamma_n$ for Player II in the game $\mathcal{G}(A_n, B_n, \alpha_n)$.
        Given $s \in S^0_n, t \in S^1_n$, let $\Gamma_{s, t} = \Gamma_{t, s}$ be the corresponding strategy for the game $\mathcal{G}((A_n)^s, (B_n)^t, \alpha_n)$. For $s, t \in S^0_n$, $(A_n)^s \cong (A_n)^t$; use this to get $\Gamma_{s, t}$, a winning strategy for Player II in the game $\mathcal{G}((A_n)^s, (A_n)^t, \infty)$. Similarly define $\Gamma_{s, t}$ for $s, t \in S^1_n$.
        
        We now describe a winning strategy $\Gamma$ for Player II in the game $\mathcal{G}(M_X, M_Y, \alpha)$.
        
        \vspace{1 mm}
        
        Case 1: suppose Player I plays $(a_0, \beta_0)$ where $a_0 = (\eta^*, 0)$ for some $\eta \in X$. Choose $n$ large enough so that $\alpha_n \geq \beta_0$. Using the density of $Y$, choose $\tau \in Y$ such that $\tau \restriction_n = \eta \restriction_n$. By back-and-forth, we can choose a tree isomorphism $F: (2^{<\omega} \cup X, \subseteq) \cong (2^{<\omega} \cup Y,\subseteq)$ such that $F$ is the identity on $2^n$. This map $F$ induces a tree isomorphism $F^*: S_X^* \to S_Y^*$ defined by $F^*(s^*) = F(s)^*$. On the first move, Player II plays $(\tau^*, 0)$. 
        
                On subsequent moves: 
        
        If Player I plays $((\nu^*, 0), \beta)$ where $\nu \in X$, then Player II plays $(F^*(\nu^*), 0)$.
        
        If Player I plays $((\nu^*, 0), \beta)$ where $\nu \in Y$, then Player II plays $((F^*)^{-1}(\nu^*), 0)$.
        
        If Player I plays $((\nu, k), \beta)$, where $(\nu,k) \in M_X(s)$ for some $s \in S^0_m \cup S^1_m$,  then Player II plays according to $\Gamma_{s, F^*(s)}$, where we take as input all the previous moves that took place in $M_X(s)$ and $M_Y(F^*(s))$. This will be valid, since either $m \leq n$, in which case $\Gamma_{s, F^*(s)}$ actually describes an isomorphism, or else $m > n$, and so the ordinals involved in the relevant previous moves will all be less than $\beta_0 \leq \alpha_n$.
        
        If Player I plays $((\nu, k), \beta)$, where $(\nu,k) \in M_Y(s)$ for some $s \in S^0_n \cup S^1_n$, then Player II plays according to $\Gamma_{(F^*)^{-1}(s), s}$, where we take as input all the previous moves that took place in $M_X((F^*)^{-1}(s))$ and $M_Y(s)$.
        
        \vspace{1 mm}
        Case 2: Suppose Player I plays $(a_0, \beta_0)$ where $a_0 \in M_X(s)$ for some $s \in S^0_N \cup S^1_N$. Choose $n \geq N$ such that $\alpha_n \geq \beta_0$. By back-and-forth, we can choose a tree isomorphism $F: (2^{<\omega} \cup X, \subseteq) \cong (2^{<\omega} \cup Y, \subseteq)$ such that $F$ is the identity on $2^n$. From $F$
        we obtain $F^*: S_X^* \to S_Y^*$ as in Case 1. On the first move, Player II plays according to $\Gamma_{s, s}$, and afterwards plays as in Case 1. 
        
        \vspace{1 mm}
        The remaining cases where Player I starts in $M_Y$ are the same, just interchange the roles of $X$ and $Y$.
		\end{claimproof}

		\medskip
		
		With both claims finished, let $X\subset 2^\omega$ be the set of sequences which are eventually zero, and $Y\subset 2^\omega$ be the set of sequences which are eventually one.  Then $M_X\equiv_\alpha M_Y$ and $M_X\not\iso M_Y$.  This completes the induction and the proof.
	\end{proof}
	
	This gives the first known example of the following behavior:
	
	\begin{cor}
		There is a complete first-order theory for whom isomorphism is neither Borel nor Borel complete.
	\end{cor}
	
	Here, the example is $\REF(\bin)$, the paradigmatic example of a superstable, non-$\omega$-stable theory.  Thus we might informally expect this behavior to be extremely common for such theories.  Since isomorphism is not Borel, we cannot truly consider $\REF(\bin)$ to be especially tame.  However, the theory is relatively simple in the sense that it cannot code much infinitary behavior.  We end with the following class of examples which follow naturally from this one:
	
	\begin{cor}\label{REFCorollary}
		For any $\alpha$ with $2\leq\alpha<\omega_1$, there is a complete first-order theory $S_\alpha$ whose isomorphism relation is not Borel, and where $T_\beta\borelleq S_\alpha$ if and only if $\beta\leq\alpha$.
		
		Each of these theories is grounded, superstable, but not $\omega$-stable.
	\end{cor}
	\begin{proof}
		Take $S_2:=\REF(\bin)$.  We construct $S_{\alpha+1}$ as $J(S_\alpha)$, and for limit $\alpha$, construct $S_\alpha$ as $\prod_\beta S_\beta$ as $\beta$ varies below $\alpha$. That $T_\alpha\borelleq S_\alpha$ follows from induction and Propositions \ref{BasicJumpProposition} and \ref{BasicProductProposition}, part (4).  That $T_{\alpha+1}\not\borelleq S_\alpha$ follows from the fact that $\|S_\alpha\|=\|T_\alpha\|<\|T_{\alpha+1}\|$, which follows from part (3) of those Propositions.
		
		Groundedness follows from part (2), and the place in the stability spectrum is a standard type-counting argument, beginning with the fact that $\REF(\bin)$ has the desired properties.
	\end{proof}
		
	We end this subsection with an open question:
	
	\begin{question}
		Let $\alpha$ be $0$ or $1$.  Is there a first-order theory $S_\alpha$ whose isomorphism relation is not Borel, and where $T_\beta\borelleq S_\alpha$ if and only if $\beta\leq\alpha$?
	\end{question}
	
	Note that the instance of the above question for $\alpha=0$ is precisely Vaught's conjecture for first-order theories. (A theory $T$ has a perfect set of nonisomorphic models if and only if $T_1 \leq_B T$.) For $\alpha =1$, abelian $p$-groups are an infinitary counterexample; we would like a first-order counterexample.

	\subsection{Infinite Branching}
	
	We now turn our attention to $\REF(\inf)$ specifically, and prove the following theorem:
	
	\begin{thm}\label{REIisBCTheorem}
		$\REF(\inf)$ is Borel complete.  Indeed, for each infinite cardinal $\lambda$, $\REF(\inf)$ is $\lambda$-Borel complete.
	\end{thm}
	
	
	\begin{proof}
    	Let $\Phi$ be the $L_{\omega_1\omega}$ sentence in the language $\{\leq\}$ describing $\omega$-trees. By Theorem~3.11 of \cite{LaskowskiShelahAleph0Stable}, $\Phi$ is $\lambda$-Borel complete for each $\lambda$, so it is enough to produce a $\lambda$-Borel reduction $f$ from $\mbox{Mod}_\lambda(\Phi)$ to $\mbox{Mod}_\lambda(\REF(\inf))$. 
        
        Call a subtree $S \subset \lambda^{<\omega}$ is {\em reasonable} if for every element $s \in S$, $\{\alpha < \lambda: s^\frown(\alpha) \not \in S\}$ is infinite. We describe an operation $S \mapsto M_S$ sending reasonable subtrees of $\lambda^{<\omega}$ of size $\lambda$, to models of $\REF(\inf)$ of size $\lambda$, such that $S \equiv_{\infty \omega} S'$ iff $M_S \equiv_{\infty \omega} M_{S'}$. It will then be routine to define a $\lambda$-Borel map $f: \mbox{Mod}_\lambda(\Phi) \to \mbox{Mod}_\lambda(\REF(\inf))$, such that given $S' \in \mbox{Mod}_\lambda(\Phi)$ there is some subtree $S \subset \lambda^{<\omega}$ reasonable with $S \cong S'$ and $f(S') \cong M_S$. Then $f$ will be the desired reduction.

 Let $I\subset\lambda^{\omega}$ be the set of all $\omega$-sequences  from $\lambda$ which are eventually zero.  For any set $M$ satisfying
 $$I\times\{0\}\subseteq M\subseteq I\times\{0,1\}$$
if we construe $M$ as an $L=\{E_n:n\in\omega\}$-structure by the rule $E_n((\eta,i),(\nu,j))$ if and only if $\eta\mr{n}=\nu\mr{n}$, then $M$ is a model of
$\REF(\inf)$.
		
		 So, given a reasonable subtree $S\subset \lambda^{<\omega}$ of size $\lambda$, let $M_S$ be the $L$-structure whose universe is
		 $$(I\times\{0\})\cup\{(\eta,1): t\frown(1)\subset\eta \ \hbox{for some $t\in\lambda^{<\omega}\setminus S$}\}$$
		
       
        We check that  the mapping $S \mapsto M_S$ works.
        
        To see this, we describe an inverse operation.  Given any $L$-structure $M$ whose universe satisfies $I\times\{0\}\subseteq M\subseteq I\times\{0,1\}$, let
        $$Tr(M)=\{s\in \lambda^{<\omega}:\forall\alpha<\lambda\exists \eta\in \lambda^\omega\left[s\frown(\alpha)\subset \eta\ \hbox{and} \  (\eta,1)\not\in M\right]\}$$
        
        We first argue that for any subtree $S\subseteq \lambda^{<\omega}$, we have $Tr(M_S)=S$.  Indeed, suppose $s\in S$.  Choose $\alpha\in\lambda$ arbitrarily.  Then
        the element $\eta:=s\frown(\alpha)\frown \overline{0}$ of $I$ witnesses that $s\in Tr(M_S)$.  Conversely, if $s\not\in S$ then as $(\eta,1)\in M_S$
        for every $\eta\supset s\frown(1)$, $s\not\in Tr(M_S)$.  
        
        Thus, in particular, $Tr(M_S)$ is a subtree of $\lambda^{<\omega}$ whenever $S$ is.  
        
         \begin{claim}
			For any subtrees $S,T$ of $\lambda^{<\omega}$, if the $L$-structures $M_S\equiv_{\infty,\omega} M_T$, then $(S,\subseteq)\equiv_{\infty,\omega} (T,\subseteq)$.
		\end{claim}
		\begin{claimproof}
		Assume $M_S\equiv_{\infty,\omega} M_T$.  Pass to a forcing extension $\V[G]$ in which $\lambda^\V$ is countable.   Choose an $L$-isomorphism $f:M_S\rightarrow M_T$.
		This induces a tree isomorphism $f^*:(Tr(M_S),\subseteq)\rightarrow (Tr(M_T),\subseteq)$.  Combined with the computation above, $(S,\subseteq)$ and $(T,\subseteq)$
		are isomorphic in $\V[G]$, so they are back-and-forth equivalent in $\V$.            
		\end{claimproof}

		To complete the proof, suppose two reasonable subtrees satisfy $(S,\subseteq) \equiv_{\infty \omega} (T,\subseteq)$. Pass to a forcing extension wherein $\lambda$ is countable, so that $S \cong T$. Then, since $S$ and $T$ are reasonable, we can choose a  tree automorphism 
		$f: (\lambda^{<\omega},\subseteq) \cong (\lambda^{<\omega},\subseteq)$ that carries $S$ to $T$. Then clearly $f$ induces an $L$-isomorphism from $M_S$ to $M_T$.  This implies that the $L$-structures $M_S$ and $M_T$
		are back-and-forth equivalent in the ground model.

	\end{proof}
	
	The following Corollary follows immediately from Corollary~\ref{BCnotshort}, Proposition~\ref{gr}, and Theorem~\ref{REIisBCTheorem}.
	
	\begin{cor}
		$\REF(\inf)$ is not short. Indeed, $\REF(\inf)$ has class-many $\equiv_{\infty\omega}$-inequivalent models in $\V$.
	\end{cor}
	

	\section{$\omega$-Stable Examples}\label{OmegaStableSection}

	Here we discuss two more first-order theories whose isomorphism relations are not Borel, but where one is Borel complete, and the other does not embed $T_3$.  Interestingly, both are extremely similar model-theoretically.  Both are $\omega$-stable with quantifier elimination, and have $\ENINDOP$ and $\eni$-depth 2, which together give a strong structure theorem in terms of stability theory.\footnote{In \cite{LaskowskiShelahAleph0Stable}, an attempt is made to characterize which first-order $\omega$-stable theories are Borel complete, using the dividing lines: ENI-DOP vs ENI-NDOP, and eni-deep vs eni-shallow. In particular, it is shown that any $\omega$-stable theory which either has ENI-DOP or is eni-deep is Borel complete; and if an $\omega$-stable theory has both ENI-NDOP and is eni-shallow, then it has fewer then $\beth_{\omega_1}$-many models up to back-and-forth equivalence.}
    Indeed, as we will show, both have exactly $\beth_2$ models up to $\equiv_{\infty\omega}$, meaning that at a ``macro'' level, they are extremely similar.  Yet at a ``micro'' level, they are quite different.  Because of the similarity of the examples, we are able to highlight exactly why one is relatively simple, while the other is not.
	
	Let us define the theories.  The first, $\K$, is due to Koerwien and constructed in \cite{KoerwienExample}.  The language has unary sorts $U$, $V_i$, and $C_i$, as well as unary functions $S_i$ and $\pi^j_i$ for $i\in \omega$ and $j\leq i+1$.  The axioms are as follows:
	
	\begin{itemize}
		\item The sorts $U$, $V_i$, and $C_i$ are all disjoint.  $U$ and each of the $V_i$ are infinite, but each $C_i$ has size 2.
		\item $\pi^{i+1}_i$ is a function from $V_i$ to $U$; $\pi^j_i$ is a function from $V_i$ to $C_j$ when $j\leq i$.
		\item For each tuple $\o c=(c_0,\ldots,c_i)$ and each $u\in U$, $\pi^{-1}_i(\o c,u)$ is nonempty.  Here $\pi_i$ refers to the product map $\pi^0_i\times\cdots\times \pi^{i+1}_i:V_i\to C_0\times \cdots C_i\times U$. 
		\item $S_i$ is a unary successor function from $V_i$ to itself, and $\pi_i\circ S_i=\pi_i$.
	\end{itemize}
	
	We have a few remarks.  Typically we will drop the subscript on $\pi_i$ and $S_i$ if it is clear from context.  There is a slight ambiguity about the sorts, whether one works in traditional first-order logic (and thus there may be ``unsorted'' elements) or in multisorted logic (where there will not be).  Since the unsorted elements never have any effect other than to complicate notation, we work in multisorted logic.
	
	The properties of $\K$ have been well studied by Koerwien  in \cite{KoerwienExample}; we summarize his findings here:
	
	\begin{thm}
		$\K$ is complete with quantifier elimination.  It is $\omega$-stable,  has $\ENINDOP$, and is $\eni$-shallow of $\eni$-depth 2.
		Furthermore, the isomorphism relation for $\K$ is not Borel.
	\end{thm}
	
	Our other theory is a tweak of $\K$, so we call it $\TK$.  The language is slightly different; we have unary sorts $U$, $V_i$, and $C_i$ as before, but have unary functions $S_i$, $\pi^0_i$, $\pi^1_i$, and $\tau_{i+1}$ for $i\in\omega$.  The axioms are as follows:
	
	\begin{itemize}
		\item The sorts $U$, $V_i$, and $C_i$ are all disjoint.  $U$ and each of the $V_i$ are infinite, but each $C_i$ has size $2^i$.
		\item $\tau_{i+1}$ is a surjection from $C_{i+1}$ to $C_i$ where, for all $c\in C_i$, $|\tau_{i+1}^{-1}(c)|=2$.
		\item $\pi^{1}_i$ is a function from $V_i$ to $U$; $\pi^0_i$ is a function from $V_i$ to $C_i$.
		\item For each tuple $c\in C_i$ and each $u\in U$, $\pi^{-1}_i(c,u)$ is nonempty.  Here $\pi_i$ refers to the product map $\pi^0_i\times \pi^{1}_i:V_i\to C_i\times U$. 
		\item $S_i$ is a unary successor function from $V_i$ to itself, and $\pi_i\circ S_i=\pi_i$.
	\end{itemize}
	
	The preceding notes also apply to $\K$.  The behavior is extremely similar, and essentially the same proofs of basic properties of $\K$ apply to $\TK$.  We summarize this now:
	
	\begin{thm}
		$\TK$ is complete with quantifier elimination.  
		It is $\omega$-stable, has $\ENINDOP$, and is $\eni$-shallow of $\eni$-depth 2.
	\end{thm}

	We can easily see that both $\K$ and $\TK$ have relatively few models up to back-and-forth equivalence:
	
	\begin{prop}  \label{submodelcount}
		$I_{\infty,\omega}(K)=I_{\infty,\omega}(TK)=\beth_2$.
		
		Indeed, every model $M$ of either theory has a submodel $N$ where $M\equiv_{\infty\omega}N$ and $|N|\leq\beth_1$.
	\end{prop}
	\begin{proof}
		Let $T$ be either $\K$ or $\TK$. For the proof of the proposition we can restrict attention to models of $T$ with a fixed algebraic closure of the empty set $\bigcup_i C_i$.  If $T = K$, then let $C$ be all finite sequences $(a_j: j < i)$ with $i > 0$ and with each $a_j \in C_j$; if $T = TK$ then let $C = \bigcup_i C_i$. 
        
        We first show $I_{\infty,\omega}(T)\geq \beth_2$.  For each $\eta\in 2^\omega$, let $u_\eta$ be some element which will eventually be part of $U$ in some model of $T$.  For any $n\in\omega$ and any $c \in C$ where $\pi^{-1}_n(c, u_{\eta})$ is nonempty, we insist the $S_n$-dimension of $\pi^{-1}_n(c, u_{\eta})$ be infinite if $\eta(n)=1$, or 
		equal to one otherwise.  (If $T = K$, then $\pi^{-1}_n(c, u_\eta)$ is nonempty if and only if $\lg(c) = n$; if $T = TK$ then $\pi^{-1}_n(c, u_\eta)$ is nonempty if and only if $c \in C_n$.) For any infinite $X\subseteq 2^\omega$, define $M_X$ to have $U^{M_X}=\{u_\eta:\eta\in X\}$ with the described behavior of the $V_i$ and $S_i$. Evidently if $Y\subseteq 2^\omega$ is infinite and $X\not=Y$, then for any $\eta\in X\setminus Y$, there is no $\nu\in Y$ where $(M_X,u_\eta)\equiv_{\infty\omega}(M_Y,u_\nu)$, and symmetrically.  Thus, $M_X\not\equiv_{\infty\omega}M_Y$.  Since there are $\beth_2$ infinite subsets of $2^\omega$, $I_{\infty\omega}(T)\geq \beth_2$.
		
		That $I_{\infty\omega}(T)\leq \beth_2$ follows immediately from the second claim.  So let $M$ be some model of $T$, of any particular cardinality.  We begin by stripping down the $V_i$.  For every $u\in U$ and $c\in C$, if $\pi^{-1}(c,u)$ is uncountable, drop all but a countable $S$-closed subset of infinite $S$-dimension.  Do this for all pairs $(c,u)$.  The result is $\equiv_{\infty\omega}$-equivalent to the original by an easy argument, and $\pi^{-1}(c,u)$ is now always countable.
		
		Next we need only ensure that $U$ has size at most continuum.  So put an equivalence relation $E$ on $U$, where we say $uEu'$ holds if and only if, for all $c\in C$, the dimensions of $\pi^{-1}(c,u)$ and $\pi^{-1}(c,u')$ are equal.  If any $E$-class is uncountable, drop all but a countably infinite subset; the resulting structure is $\equiv_{\infty\omega}$-equivalent to the original again.  Further, each $E$-class is now countable, and there are only $|C^\omega|=\beth_1$ possible $E$-classes, so the structure now has size at most $\beth_1$.  This completes the proof.
	\end{proof}

%
	
	Any additional complexity of either theory comes from elementary permutations of the algebraic closure of the empty set.   In any model $M$ of either $\K$ or $\TK$, $\acl_M(\emptyset)=\bigcup_{i\in\omega} C_i(M)$.  In models $M$ of $\TK$, the projection functions $\{\tau_i\}$ naturally induce a tree structure, so we think of $\acl_M(\emptyset)$ as being a copy of $(2^{<\omega},\le)$.  In models $M$ of $\K$, as each $C_i(M)$ has exactly two elements, so one can think of $\acl_M(\emptyset)$ as being indexed by $2\times\omega$. Note, however, there is some freedom with all this; for our purposes, $\acl_M(\emptyset)$ could equally well be viewed as any subset of $\acl_{M^{eq}}(\emptyset)$ whose definable closure contains $acl_M(\emptyset)$ (here $M^{eq}$ is the result of eliminating imaginaries from $M$). In the case $M \models K$ it is most convenient to say that $\acl_M(\emptyset)$ is all finite sequences $\<a_j:j<i\>$, where each $a_j\in C_j(M)$.  These finite sequences, when ordered by initial segment, also give a natural correspondence of $\acl_M(\emptyset)$ with the tree $(2^{<\omega},\le)$.  Henceforth, when discussing models $M$ of either $\K$ or $\TK$, we will view $\acl_M(\emptyset)$ as being indexed by the tree $(2^{<\omega},\le)$.
	
	Next, we discuss the group $G$ of elementary permutations of $\acl_M(\emptyset)$ (which only depends on the theory).
    For $\K$, the relevant group is $(2^\omega,\oplus)$, the direct product of $\omega$ copies of the two-element group. Indeed,  in any model of $\K$, any elementary permuation of $\acl_M(\emptyset)$ is determined by the sequence of permutations of $C_i(M)$.
    In $\TK$,  as elementary permutations just have to respect the $\tau_i$ structure,  the relevant group of elementary permutations  is $\Aut(2^{<\omega}, \leq)$. 
    Both of these groups are compact (in fact this is true for all first order theories), but only the group for $\K$ is abelian.  It turns out that being abelian is enough to produce relative simplicity, while being nonabelian leaves enough room  to allow  $\TK$ to be Borel complete.

\bigskip

For the next proposition we need some setup.

Let $X$ be the Polish space of all $f: 2^{<\omega} \to (\omega+1 \backslash \{\emptyset\})$. Let $T$ be either $\K$ or $\TK$, and let $G$ be either $(2^\omega,\oplus)$ or $\Aut(2^{<\omega},\leq)$, respectively. $G$ acts on $2^{<\omega}$ naturally: if $G=(2^\omega,\oplus)$, then $g\cdot \sigma=g\mr{|\sigma|}\oplus \sigma$.  If $G=\Aut(2^{<\omega},\leq)$, then $g\cdot \sigma=g(\sigma)$. From this we get an action of $G$ on $X$: namely for $f \in X, g \in G$, $(g \cdot f)(\sigma)  = f(g^{-1} \cdot \sigma)$. This is a strongly definable, continuous action in the sense of Definition~\ref{sdPolishSpace}. Let $E_G$ be the equivalence relation on $X$ induced by the action, as well as the equivalence relation on $\mathcal{P}_{\aleph_1}(X)$ induced by the diagonal action.

Now $G$ acts diagonally on $X^\omega$ also; this action commutes with the permuation action of $S_\infty$ on $X^\omega$. So $G \times S_\infty$ acts naturally on $X^\omega$; let $E_{G \times S_\infty}$ be the equivalence relation induced by this action.
\begin{prop}  \label{relate}  Let $T$ be either $\K$ or $\TK$. Then:
		
		\begin{itemize}
			\item $(\Mod(T),\iso)\boreleq (X^\omega,E_{G\times S_\infty})$.
			\item $(\Mod(T),\iso) \sim_{\HC}(\m P_{\aleph_1}(X),E_G)$.
		\end{itemize}
		
	\end{prop}
	\begin{proof}
    	For the various codings below, fix a pairing function $\langle \cdot, \cdot \rangle: (\omega+1 \backslash \emptyset)^2 \to (\omega+1 \backslash \emptyset)$. Note that one difference between $\K$ and $\TK$ that frequently affects the coding is: $\pi^{-1}(\emptyset, u)$ is only nonempty for models of $\TK$.

		To show $(\Mod(T),\iso)\borelleq (X^\omega,E_{G\times S_\infty})$, first let $M\in\Mod(T)$ be arbitrary.  We may choose an indexing of $\acl_M(\emptyset)$ by $2^{<\omega}$ ,  and of $U^M$ by $\omega$, using the original indexing of the universe of $M$ by $\omega$. Then each element $u\in U^M$ induces a function $c_u \in X$, where $c_u(\sigma)$ is the $S$-dimension of $\pi^{-1}(\sigma, u)$. (In the case of $T = \K$, define $c_u(\emptyset) = 1$.)  Then take $M$ to the sequence $(c_{u_n}:n\in\omega)$, where $u_n$ is the $n$-th element of $U$. It is clear that this works.

		To show $(X^\omega,E_{G\times S_\infty})\borelleq (\Mod(T),\iso)$, fix a sequence $\overline{x}=(x_n:n\in\omega)$.  We describe the case for the theory $\K$.  We define $M_{\overline{x}}$ to have $U^{M_{\overline{x}}}=\{4n: n \in \omega\}$, and have $C^{M_{\overline{x}}}_i = \{4i+1, 4i+2\}$.  Then, using the infinitely many remaining elements, we arrange that for each $\sigma \in 2^{<\omega} \backslash \emptyset$ and each $n < \omega$, the $S$-dimension of $\pi^{-1}(\sigma_*, 4n)$ is $\langle x_n(\sigma), x_n(\emptyset) \rangle$, where $\sigma_* = (4i+1+\sigma(i): i < \lg(\sigma))$. The case for $\TK$ is similar.
		
        We have shown that $(\mbox{Mod}(T), \cong) \sim_B (X^\omega, E_{G \times S_\infty})$. It follows that they are $\hcleq$-biembeddable; so to conclude the proof of the proposition it suffices to show that $(X^\omega, E_{G \times S_\infty}) \sim_{\HC} (\mathcal{P}_{\aleph_1}(X), E_G)$.
        
		To show $(X^\omega,E_{G\times S_\infty})\hcleq (\m P_{\aleph_1}(X),E_G)$, we just need to handle multiplicities.  So fix $\overline{x}= (x_n: n \in \omega) \in X^\omega$, and for each $n$, define $m_{\overline{x}}(n)$ to be $|\{m: x_m = x_n\}|$. For each $n\in\omega$ let $y_n \in X$ be defined by: $y_n(\sigma) = \langle x_n(\sigma), m_{\overline{x}}(n) \rangle$. Then $\overline{x} \mapsto \{y_n: n < \omega\}$ works.
		
		We define a reverse embedding $f: (\m P_{\aleph_1}(X),E_G)\hcleq (X^\omega, E_{G \times S_\infty})$ as follows (where recall that we do not require $f$ to be single-valued). Namely, given $A \subset X$ countable and given $\overline{x} \in X^{\omega}$, put $(A, \overline{x}) \in f$ whenever $\overline{x}$ is an infinite-to-one enumeration of $A$. Also put $(\emptyset, \overline{x}) \in f$ for some fixed injective $\overline{x} \in X^\omega$.
	\end{proof}
	

	\subsection{Koerwien's Example}
	
	For this subsection, we want to prove that $\K$ is not Borel complete, and further, to characterize exactly which $T_\alpha$ embed in $\K$. To do this, we show directly that $T_2\borelleq \K$, and then show that $||K||_{\ptl} = \beth_2$.  The former is quite straightforward:
	
	\begin{prop}
		$T_2\borelleq\K$.
	\end{prop}
	\begin{proof}
		Let $X\subset 2^\omega$ be countable; we describe a model $M_X\models\K$ from which $X$ can be easily recovered.  Let $U$ be the set $A\cup X$, where $A$ is some countable infinite set which is disjoint from $X$.  Let $C_i = \{c^i_0, c^i_1\}$.  For each tuple $(a,\o c)$ with $a\in A$, arrange that $\pi^{-1}(a, \overline{c})$ has $S$-dimension 1.  For each tuple $(x,\o c)$ with $x\in X$, arrange that $\pi^{-1}(x, \overline{c})$ has $S$-dimension $x(|\overline{c}|)+2$. Clearly $M_X \cong M_Y$ iff $X = Y$.
		
		Now it is not hard, given $\overline{x} = (x_n: n \in \omega) \in (2^\omega)^\omega$, to produce in a Borel fashion a model $M_{\overline{x}} \models \K$ with universe $\omega$, such that $M_{\overline{x}} \cong  M_{\{x_n: n \in \omega\}}$. This gives a Borel reduction from $((2^\omega)^\omega, F_2)$ to $(\mbox{Mod}(\K), \cong)$, which suffices (see the discussion preceding Theorem~\ref{REFCodesIso2Theorem}).
	\end{proof}

	Having accomplished this, we can state everything we need about $\K$:
	
	\begin{thm}
		$||K||=\beth_2$.  Therefore, $\K$ is not Borel complete; indeed, there is no Borel embedding of $T_3$ into $\Mod(\K)$.
			\end{thm}
	\begin{proof}
	        That $||K||\ge\beth_2$ follows immediately from Proposition~\ref{submodelcount}.  
		Since $G=(2^\omega,\oplus)$ is compact and abelian, 
		$||(\P_{\aleph_1}(X),E_G)||\le\beth_2$ by Theorem~\ref{compact2}.
		As $||\K||=||(\P_{\aleph_1}(X),E_G)||$ by Proposition~\ref{relate}, we conclude that $||K||=\beth_2$.  That there is no Borel embedding of $T_3$ into $\Mod(\K)$
		is immediate from Theorem~\ref{translate}(2) and Corollary~\ref{countT}.
	\end{proof}
	
	Once we have one such example, we can apply the usual constructions to get a large class of $\omega$-stable examples:

	\begin{cor}
		For each ordinal $\alpha$ with $2\leq\alpha<\omega_1$, there is an $\omega$-stable theory $S_\alpha$ whose isomorphism relation is not Borel and where $T_\beta\borelleq S_\alpha$ if and only if $\beta\leq\alpha$.
	\end{cor}
	\begin{proof}
		Let $S_2$ be $\K$.  Then proceed inductively as in Corollary~\ref{REFCorollary}.
	\end{proof}
	
	There is no such example when $\alpha=0$, since Vaught's Conjecture holds for $\omega$-stable theories. ($T$ has a perfect set of nonisomorphic models iff $T_1 \leq_B T$, and so whenever $T$ is $\omega$-stable, either $T \leq_B T_0$ or $T_1 \leq_B T$.) It is unknown if there is an example when $\alpha=1$.

	\subsection{A New $\omega$-Stable Theory}


	We now consider $\TK$, with the aim of showing it is Borel complete.  Indeed with Proposition~\ref{relate} we have already shown $(\Mod(\TK),\iso)$ is Borel equivalent to $(X^\omega,E_{G\times S_\infty})$, where $X$ is the space of all $c: 2^{<\omega} \to \omega$. (We are replacing $\omega+1 \backslash \emptyset$ with $\omega$, which is harmless.) Recall that $G=\Aut(2^{<\omega},\leq)$ acts on $X$ by permuting the fibers; that is, for any $c:2^{<\omega}\to \omega$, any $g\in G$, and any $\sigma \in 2^{<\omega}$, $(g\cdot c)(\sigma)=c(g^{-1}\cdot \sigma)$.  Then $G$ acts on $X^\omega$ diagonally, while $S_\infty$ acts on $X^\omega$ by permuting the fibers, so these actions commute with one another and induce an action of the product group $G\times S_\infty$.  
    
    Thus, to show $\TK$ is Borel complete, it is enough to show $(X^\omega, E_{G\times S_\infty})$ is Borel complete, which we do directly.  
	
	\begin{thm}  \label{graphs}
		$(\textrm{Graphs},\iso)\borelleq (X^\omega,E_{G\times S_\infty})$.
	\end{thm}
	\begin{proof}
		To simplify notation, for the whole of this proof we write $E$ in place of   $E_{G\times S_\infty}$.	
		We need some setup first.  Observe that $G$ naturally acts on $2^\omega$, the set of branches of $(2^{<\omega},\le)$, by $g\cdot\sigma=\bigcup_n g\cdot \sigma\mr{n}$; this is a well-defined sequence precisely because $g$ is a tree automorphism.  Let $\{D_i:i\in\omega\}$ be a countable set of dense, disjoint, countable subsets of $2^\omega$, and let $D=\bigcup_i D_i$. We need one claim, where we use the relative complexity of $G$ (it would not go through if we replaced $\TK$ with $\K$):

		\begin{claim}
			For any $\sigma\in S_\infty$, there is a $g\in G$ where for all $i\in\omega$, $g\cdot D_i=D_{\sigma(i)}$ as sets.
		\end{claim}
		\begin{claimproof}
			We construct such a $g$ by a back-and-forth argument.  So let $\m F$ be the set of finite partial functions from $D$ to itself, satisfying all the following:
			
			\begin{itemize}
				\item For each $f\in \m F$ and each $\eta\in \mbox{dom}(f)$, if $\eta\in D_i$, then $f(\eta)\in D_{\sigma(i)}$.
				\item For each $f\in \m F$ and each $\eta,\nu\in \mbox{dom}(f)$,  $\lg(\eta\wedge\nu)=\lg(f(\eta)\wedge f(\nu))$, where 
				$\eta\wedge\nu$ denotes the longest common initial segment of  $\eta$ and $\nu$.
				\item The previous conditions, but with $f^{-1}$ and $\sigma^{-1}$ instead of $f$ and $\sigma$.
			\end{itemize}


			
			Once we establish that  $\m F$ is a back-and-forth system, then $\m F$ defines a $g\in G$ with the desired property.  For choose a bijection $f: D \to D$ such that the finite restrictions of $f$ all lie in $\mathcal{F}$.  If $s\in 2^n$, let $g(s)$ be $f(\eta)\mr{n}$ for any $\eta$ extending $s$; because of the consistency properties of $\m F$, and because $D$ is dense, this is well-defined.  Then clearly $g \in G$ has the desired property with respect to $\sigma$.
			
			So we need only show that $\m F$ is a back-and-forth system. Of course the empty function is in $\m F$. So say $f\in\m F$ and $\eta\in 2^\omega$; we want $f'\supset f$ in $\m F$ with $\eta\in\dom(f')$. The case where $f$ is empty is easy, so suppose $f$ is nonempty. We also assume $\eta\not\in\dom(f)$ already.  Let $n$ be maximal among $\{\lg(\eta\wedge \nu):\nu\in\dom(f)\}$, and let $\nu\in\dom(f)$ be such that $\lg(\eta\wedge\nu)=n$.  We then pick an element $f(\eta)$ of $2^\omega$ which extends $f(\nu)\mr{n}\frown (1-f(\nu)(n))$.  That is, $f(\eta)$ agrees with $f(\nu)$ before stage $n$, but disagrees with it at $n$.  If $\eta\in D_i$, choose this element from $D_{\sigma(i)}$, which is possible by density. This clearly satisfies the desired properties, and the other direction is symmetric, proving the claim.
		\end{claimproof}
		
		\medskip

			Given $\eta,\tau\in 2^\omega$ and $k\in\{1,2,3\}$, let $c^k_{\eta,\tau}:2^{<\omega}\to\omega$ be the coloring which sends $s\in 2^{<\omega}$ to $k$, if $s\subset \eta$ or $s\subset \tau$, or $0$ otherwise. Also, fix a bijection $\rho: \omega\to \bigcup_{i \leq j} D_i \times D_j$. We have now fixed enough notation and can describe our map
			$f:\textit{Graphs}\rightarrow X^\omega$.
		
		Let $R$ be a graph on $\omega$ -- that is, $R$ is a binary relation on $\omega$ which is symmetric and irreflexive.  For each $n\in\omega$, $\rho(n)$ is a pair $(\eta,\tau)\in D_i\times D_j$ for some $i \leq j$.  If $i=j$, define $c_n=c^1_{\eta,\tau}$.  If $i<j$ and $\{i,j\}\in R$, then let $c_n=c^2_{\eta,\tau}$.  Otherwise let $c_n=c^3_{\eta,\tau}$.  Then 
		put $f(R):=(c_n:n\in\omega)$.  $f(R)$ is a visibly element of $X^\omega$, and clearly $f$ is Borel. Note also that $f$ is injective.
		
		Suppose $\sigma:(\omega, R) \cong (\omega, R')$ is a graph isomorphism. We show that $f(R) E f(R')$. By the claim, there is a $g\in G$ where for all $i\in\omega$, $g\cdot D_i=D_{\sigma(i)}$. 
		Let $A$ be the range of $f(R)$ and let $A'$ be the range of $f(R')$. 
		We show that $g \cdot A = A'$ setwise. First suppose $c^1_{\eta,\tau} \in A$. Let $i$ be such that $\eta,\tau\in D_i$, so $g(\eta),g(\tau)\in D_{\sigma(i)}$.  Then $g\cdot c^1_{\eta,\tau}=c^1_{g(\eta),g(\tau)} \in A'$.  Similarly if $c^2_{\eta,\tau}\in A$, there is some $i<j$ where $\eta\in D_i$, $\tau\in D_j$, and $\{i,j\}\in R$.  Since $\sigma:(\omega,R)\to 
		(\omega, R')$ is a graph isomorphism, $\{\sigma(i),\sigma(j)\}\in R'$, so $c^2_{g(\eta),g(\tau)}\in A'$ (this uses $c^2_{g(\eta), g(\tau)} = c^2_{g(\tau), g(\eta)}$). The case $c^3_{\eta,\tau} \in A$ is the same.  Thus $g \cdot A \subset A'$; by a symmetric argument $g \cdot A = A'$. Since $g \cdot f(R)$ and $f(R')$ are both injective and they have the same range, some permutation of $g\cdot f(R)$ is equal to $f(R')$, i.e. $f(R)E f(R')$.
		
		It only remains to show that if $f(R) E f(R')$, then $(\omega,R)\iso (\omega,R')$. 
		So suppose $f(R) E f(R')$. Let $A$ be the range of $f(R)$ and let $A'$ be the range of $f(R')$, and choose $g \in G$ such that $g\cdot A = A'$. Let $i < \omega$; then since for all $\eta, \tau \in D_i$, $c^1_{g(\eta), g(\tau)} \in A'$, we have that $g \cdot D_i = D_{\sigma(i)}$ for some $\sigma(i) < \omega$. I claim that $\sigma: (\omega, R) \cong (\omega, R')$. Indeed, for $i < j$, $(i, j) \in R$ iff there are $\eta \in D_i, \tau \in D_j$ with $c^2_{\eta, \tau} \in A$, which is the case iff there are $\eta \in D_{\sigma(i)}, \tau \in D_{\sigma(j)}$ with $c^2_{\eta, \tau} \in A'$, which is the case iff $(i, j) \in R'$. 
	\end{proof}

	We have now shown:
	
	\begin{thm}
		$\TK$ is Borel complete.
	\end{thm}
    
    \begin{proof}
    By Theorem~\ref{graphs}, together with the  fact that graphs are Borel complete.
    \end{proof}
	
	This resolves a few open questions, raised in \cite{LaskowskiShelahAleph0Stable}:
	
	\begin{cor}
		The $\omega$-stable theory $TK$  is Borel complete, but does not have $\ENIDOP$ and is not $\eni$-deep.  Indeed $TK$ is not $\lambda$-Borel complete for any $\lambda$ with $2^\lambda>\beth_2$, as $|CSS(TK)_\sat|=\beth_2$.
	\end{cor}

		\bibliography{Citations}

\begin{thebibliography}{10}

\bibitem{BarwiseScottSentences}
Jon Barwise.
\newblock Back and forth through infinitary logic.
\newblock In Jon Barwise, editor, {\em Studies in Model Theory}, volume~8 of
  {\em Studies in Mathematics}, pages 5--34. Mathematical Association of
  America, Buffalo, 1973.

\bibitem{barwise2}
Jon Barwise.
\newblock {\em Admissible sets and structures}.
\newblock Perspectives in mathematical logic. Springer Berlin Heidelberg, 1975.

\bibitem{FriedmanStanleyBC}
Harvey Friedman and Lee Stanley.
\newblock A {Borel} reducibility theory for classes of countable structures.
\newblock {\em The Journal of Symbolic Logic}, 54:894--914, 1989.

\bibitem{GaoIDST}
Su~Gao.
\newblock {\em Invariant Descriptive Set Theory}.
\newblock Chapman \& Hall/CRC Pure and Applied Mathematics. CRC Press, 2008.

\bibitem{hjorthBook}
Greg Hjorth.
\newblock {\em Classification and Orbit Equivalence Relations}, volume~75 of
  {\em Mathematical Surveys and Monographs}.
\newblock American Mathematical Society, 1999.

\bibitem{HjorthKechrisLouveau}
Greg Hjorth, Alexander Kechris, and Alain Louveau.
\newblock {Borel} equivalence relations induced by actions of the symmetric
  group.
\newblock {\em Annals of Pure and Applied Logic}, 92:63--112, 1998.

\bibitem{HodgesModelTheory}
Wilfrid Hodges.
\newblock {\em Model Theory}.
\newblock Encyclopedia of Mathematics and its Applications. Cambridge
  University Press, 1993.

\bibitem{JechSetTheory}
Thomas Jech.
\newblock {\em Set Theory: The Third Millennium Edition, revised and expanded}.
\newblock Springer Monographs in Mathematics. Springer Berlin Heidelberg, 2006.

\bibitem{LevyAbs}
Akihiro Kanamori.
\newblock Levy and set theory.
\newblock {\em Annals of Pure and Applied Logic}, 140:233--252, 2006.

\bibitem{kaplan}
Itay Kaplan and Saharon Shelah.
\newblock Forcing a countable structure to belong to the ground model.
\newblock {\em Mathematical Logic Quarterly}, 2015.
\newblock Accepted.

\bibitem{KechrisDST}
Alexander Kechris.
\newblock {\em Classical Descriptive Set Theory}.
\newblock Graduate Texts in Mathematics. Springer New York, 1 edition, 1995.

\bibitem{Keisler}
H.~Jerome Keisler.
\newblock {\em Model Theory for Infinitary Logic}.
\newblock North Holland, 1971.

\bibitem{KoerwienBRDepth}
Martin Koerwien.
\newblock Comparing {B}orel reducibility and depth of an $\omega$-stable
  theory.
\newblock {\em Notre Dame Journal of Formal Logic}, 50:365--380, 2009.

\bibitem{KoerwienExample}
Martin Koerwien.
\newblock A complicated $\omega$-stable depth 2 theory.
\newblock {\em Journal of Symbolic Logic}, 76:47--65, 2011.

\bibitem{LaskowskiShelahAleph0Stable}
Michael~C. Laskowski and S.~Shelah.
\newblock {B}orel completeness of some $\aleph_0$-stable theories.
\newblock {\em Fundamenta Mathematicae}, 2013.

\bibitem{MarkerMT}
David Marker.
\newblock {\em Model theory : an introduction}, volume 217 of {\em Graduate
  texts in mathematics}.
\newblock Springer, New York, 2002.

\bibitem{ZapletalForcingBorel}
Jindrich Zapletal.
\newblock {\em Forcing Borel reducibility invariants}.
\newblock August, 2013.

\end{thebibliography}
\end{document}